\documentclass[12pt]{amsart}
\usepackage{amssymb,amsmath,amscd}
\newtheorem{thm}{\sc Theorem}[section]
\newtheorem{prop}[thm]{\sc Proposition}

\newtheorem{cor}[thm]{\sc Corollary}
\theoremstyle{definition}\newtheorem{exa}[thm]{\sc Example}
\theoremstyle{definition}\newtheorem{de}[thm]{\sc Definition}
\theoremstyle{definition}\newtheorem{rem}[thm]{\sc Remark}
\theoremstyle{definition}

\numberwithin{equation}{section}

\DeclareMathOperator{\R}{\mathbb R}

\DeclareMathOperator{\C}{\mathbb C}

\begin{document}
\title[Embeddability for CR-Manifolds]{Embeddability for three-Dimensional Cauchy-Riemann Manifolds and CR Yamabe invariants}
\author[S. Chanillo, H.-L.~Chiu, P. Yang]{Sagun Chanillo, Hung-Lin~Chiu and Paul Yang}

\address{S. Chanillo, Department of Mathematics, Rutgers University, 110 Frelinghuysen Rd., Piscataway, NJ 08854, U.S.A.}
\email{chanillo@math.rutgers.edu}

\address{H.-L. Chiu, Department of Mathematics, National Central University, Chung Li, 32054, Taiwan, R.O.C.}
\email{hlchiu@math.ncu.edu.tw}

\address{P. Yang, Department of Mathematics, Princeton University, Princeton, NJ 08544, U.S.A.}
\email{Yang@math.princeton.edu}

\renewcommand{\subjclassname}{ AMS
Classification number\textup{(2010)}}

\begin{abstract}
Let $M^{3}$ be a closed CR 3-manifold. In this paper, we derive a Bochner formula for the Kohn Laplacian in which the pseudohermitian
torsion doesn't play any role. By means of this formula we show that the nonzero eigenvalues of the Kohn Laplacian have a lower bound,
provided that the CR Paneitz operator is nonnegative and the Webster curvature is positive.
This means that $M^{3}$ is embeddable when the CR Yamabe constant is positive and the CR Paneitz operator is nonnegative. Our lower bound estimate is
sharp. In addition, we show that the embedding is stable in the sense of Burns and Epstein. Lastly we show that the CR Paneitz operator for embeddable CR
structures given by polynomial deformations and close to the standard CR structure on $S^3$ is positive on the subspace of spherical harmonics
$\oplus_{p\geq 1} H_{p,0}\oplus H_{0,p}$.
\end{abstract}

\maketitle

\section{Introduction,Statements and Notation}
It is well known that the embedding problem remains open for three
dimensional CR manifolds. In contrast, in the higher dimensional
case, any strictly pseudoconvex, closed CR manifold can always be
realized as an embedding in some $\C^{n}$ (see \cite{BD}). In the
3-dimensional case, there exists nonembeddable examples \cite{AS,R},
and in general generic CR structures are not embeddable (see
\cite{B} and \cite{BE}).

In his paper \cite{Le1}, L. Lempert asked two fundamental questions about the embeddability problem.
The first one is related to the closedness property of
CR structures and the second one is to the stability property. In this paper we address Lempert's questions.

Throughout this paper, we will use the notations and terminology in (\cite{L}) unless otherwise specified.
Let $(M,J,\theta)$ be a closed three-dimensional pseudo-hermitian manifold,
where $\theta$ is a contact form and $J$ is a CR structure compatible with the contact bundle $\xi=\ker\theta$.
The CR structure $J$ decomposes $\bf \C\otimes\xi$ into the direct sum of $T_{1,0}$ and $T_{0,1}$ which are eigenspaces of $J$ with respect to $i$ and $-i$, respectively.
The Levi form $\left\langle\  ,\ \right\rangle_{L_\theta}$ is the Hermitian form on $T_{1,0}$ defined by $\left\langle Z,W\right\rangle_{L_\theta}=-i\left\langle d\theta,Z\wedge\overline{W}\right\rangle$.
We can extend $\left\langle\  ,\ \right\rangle_{L_\theta}$ to $T_{0,1}$ by defining $\left\langle\overline{Z} ,\overline{W}\right\rangle_{L_\theta}=\overline{\left\langle Z,W\right\rangle}_{L_\theta}$ for all $Z,W\in T_{1,0}$.
The Levi form induces naturally a Hermitian form on the dual bundle of $T_{1,0}$, denoted by $\left\langle\  ,\ \right\rangle_{L_\theta^*}$,
and hence on all the induced tensor bundles. Integrating the hermitian form (when acting on sections) over $M$ with respect to the volume form $dV=\theta\wedge d\theta$,
we get an inner product on the space of sections of each tensor bundle.
We denote the inner product by the notation $\left\langle \ ,\ \right\rangle$. For example
\begin{equation}\label{21}
\left\langle\varphi ,\psi\right\rangle=\int_{M}\varphi\bar{\psi}\ dV,
\end{equation}
for functions $\varphi$ and $\psi$.

Let $\left\{T,Z_1,Z_{\bar{1}}\right\}$ be a frame of $TM\otimes \bf C$, where $Z_1$ is any local frame of $T_{1,0},\  Z_{\bar{1}}=\overline{Z_1}\in T_{0,1}$ and $T$ is the characteristic vector field, that is, the unique vector field such that $\theta(T)=1,\ d\theta(T,\cdot)=0$. Then $\left\{\theta,\theta^1,\theta^{\bar 1}\right\}$, the coframe dual to $\left\{T,Z_1,Z_{\bar{1}}\right\}$, satisfies
\begin{equation}\label{22}
d\theta=ih_{1\bar 1}\theta^1\wedge\theta^{\bar 1}
\end{equation}
for some positive function $h_{1\bar 1}$. We can always choose $Z_1$
such that $h_{1\bar 1}=1$; hence, throughout this paper, we assume
$h_{1\bar 1}=1$

  The pseudohermitian connection of $(J,\theta)$ is the connection
$\nabla$ on $TM\otimes \bf C$ (and extended to tensors) given in terms of a local frame $Z_1\in T_{1,0}$ by

\begin{equation*}
\nabla Z_1=\theta_1{}^1\otimes Z_1,\quad
\nabla Z_{\bar{1}}=\theta_{\bar{1}}{}^{\bar{1}}\otimes Z_{\bar{1}},\quad
\nabla T=0,
\end{equation*}

where $\theta_1{}^1$ is the $1$-form uniquely determined by the following equations:

\begin{equation}\label{id10}
\begin{split}
d\theta^1&=\theta^1\wedge\theta_1{}^1+\theta\wedge\tau^1\\
\tau^1&\equiv 0\mod\theta^{\bar 1}\\
0&=\theta_1{}^1+\theta_{\bar{1}}{}^{\bar 1},
\end{split}
\end{equation}

where $\theta_{1}{}^{1}$ and $\tau^1$ are called the connection form and the pseudohermitian torsion, respectively.
Put $\tau^1=A^1{}_{\bar 1}\theta^{\bar 1}$.
The structure equation for the pseudohermitian connection is

\begin{equation*}
d\theta_1{}^1=R\theta^1\wedge\theta^{\bar 1}
+2iIm(A^{\bar{1}}{}_{1,\bar{1}}\theta^1\wedge\theta),
\end{equation*}

where $R$ is the Tanaka-Webster curvature.

We will denote components of covariant derivatives with indices preceded by a comma; thus we write $A^{\bar{1}}{}_{1,\bar{1}}\theta^1\wedge\theta$.
The indices $\{0, 1, \bar{1}\}$ indicate derivatives with respect to $\{T, Z_1, Z_{\bar{1}}\}$.
For derivatives of a scalar function, we will often omit the comma,
for instance, $\varphi_{1}=Z_1\varphi,\  \varphi_{1\bar{1}}=Z_{\bar{1}}Z_1\varphi-\theta_1^1(Z_{\bar{1}})Z_1\varphi,\  \varphi_{0}=T\varphi$ for a (smooth) function.

Next we consider several natural differential operators occuring in this paper. For a detailed description, we refer the reader to the article \cite{L}.
For a smooth function $\varphi$, the Cauchy-Riemann operator $\partial_{b}$ can be defined locally by
\[\partial_{b}\varphi=\varphi_{1}\theta^{1},\]
and we write $\bar{\partial}_{b}$ for the conjugate of $\partial_{b}$. A function $\varphi$ is called CR holomorphic if $\bar{\partial}_{b}\varphi=0$.
the divergence operator $\delta_b$ takes $(1,0)$-forms to functions by $\delta_b(\sigma_{1}\theta^{1})=\sigma_{1,}{}^{1}$,
and similarly, $\bar{\delta}_b(\sigma_{\bar 1}\theta^{\bar 1})=\sigma_{\bar 1,}{}^{\bar 1}$.

If $\sigma =\sigma_{1}\theta^{1} $ is compactly supported, Stokes' theorem applied to the 2-form $\theta\wedge\sigma $ implies the divergence formula:
\[\int_{M}\delta_{b}\sigma \theta\wedge d\theta=0.\]
It follows that the formal adjoint of $\partial_{b}$ on functions with respect to the Levi form and the volume element $\theta\wedge d\theta$ is
$\partial_{b}^{*}=-\delta_{b}$.
The Kohn Laplacian on functions determined by $\theta$ is
\[\Box_{b}=2\bar{\partial}_{b}^{*}\bar{\partial_{b}},\]

Define $P\varphi=(\varphi_{\bar{1}}{}^{\bar{1}}{}_1+iA_{11}\varphi^1)\theta^1$ (see \cite{L})
which is an operator that characterizes CR-pluriharmonic functions,
and $\overline{P}\varphi=(\varphi_{1}{}^{1}{}_{\bar{1}}-iA_{\bar{1}\bar{1}}\varphi^{\bar{1}})\theta^{\bar{1}}$, the conjugate of $P$.
The CR Paneitz operator $P_0$ is defined by

\[ P_0\varphi=\delta_b(P\varphi).\]

More explicitly,
\begin{equation*}
\begin{split}
P_{0}f&=\frac{1}{4}(\Box_{b}\overline{\Box}_{b}f-4i(A^{11}f_{1})_{1})\\
&=\frac{1}{8}\big((\overline{\Box}_{b}\Box_{b}+\Box_{b}\overline{\Box}_{b})f+8Im(A^{11}f_{1})_{1}\big).
\end{split}
\end{equation*}
It follows that $P_{0}$ is a real and symmetric operator.

\begin{de}
The Paneitz operator $P_{0}$ is nonnegative if
\[\int_{M}(P_{0}\varphi)\bar{\varphi}\geq 0,\]
for all smooth functions $\varphi$.
\end{de}
Note that the nonnegativity of $P_{0}$ is a CR invariant in the
sense that it is independent of the choice of the contact form
$\theta$. This follows by observing that if
$\widetilde{\theta}=e^{2f}\theta$ be another contact form, we have
the following transformation laws for the volume form and the CR
Paneitz operator respectively (see Lemma 7.4 in \cite{H}):
\[\widetilde{\theta}\wedge d\widetilde{\theta}=e^{4f}\theta\wedge d\theta;\ \ \ \ \ \  \widetilde{P_{0}}=e^{-4f}P_{0}.\]

We also observe that when the Webster torsion $A_{11}\equiv 0$, then the Paneitz operator $P_{0}$ is given by,
\[P_{0}=\frac{1}{4}\Box_{b}\overline{\Box}_{b}.\]
Thus the vanishing of torsion implies that $P_{0}\geq 0$. We also recall that the vanishing of torsion is equivalent to $L_{T}J=0$ where $L$
is the Lie derivative.

In the higher dimensional case, there exists an analog of $P_{0}$
which satisfies the covariant property. In this case, Graham and
Lee, in \cite{GL}, had shown the nonnegativity of $P_{0}$.

\begin{de}
Suppose that $\widetilde{\theta}=e^{2f}\theta$. The CR Yamabe constant is defined by\\
$\inf_{\widetilde{\theta}}{\{\int_{M}\widetilde{R}\ \widetilde{\theta}\wedge d\widetilde{\theta} : \int\widetilde{\theta}\wedge d\widetilde{\theta}=1\}}$.
\end{de}
The CR Yamabe constant is a CR invariant. We are now in a position to describe our main theorems. In this paper we show

\begin{thm}\label{main1}
Let $M^{3}$ be a closed CR manifold. If $P_{0}\geq 0$ and $R>0$, then the non-zero eigenvalues $\lambda$ of $\Box_{b}$ satisfy $\lambda\geq\min{R}$, hence the range of $\Box_{b}$
is closed. If $P_{0}\geq 0$ and the CR Yamabe constant $>0$, then $M^{3}$ can be embedded into $\C^{n}$, for some $n$.
\end{thm}

\begin{rem}
The fact $\Box_{b}$ has closed range is equivalent to global embedding is a result of Kohn (\cite{K})
\end{rem}
In section 3 we prove the stability theorem:

\begin{thm}\label{main2}
Under (\ref{se1}), (\ref{se4}), (\ref{se5}), the embedding is stable which means that if $|t|$ is small enough then the CR embedding $\Psi_{t}$ is close to $\Psi_{0}$.
\end{thm}

\begin{rem}
If $M$ is embeddable then the Paneitz operator $P_{0}$ has closed range (see \cite{CC}).
\end{rem}

We now turn our attention to section 4, where we show theonem \ref{main1} has a conuerse. We shall also compare our result with those of Burns-Epstein \cite{BE} and Bland \cite{B}.          .
Let $(M,J,\theta )$ be a CR structure. Let $\phi $ be a complex valued smooth function on M, such that $\|\phi\|_{\infty} <1$.
For  $\theta$ fixed consider a deformation of the CR structure given by
\[Z_{\bar{1}}^{\phi} =Z_{\bar{1}} + \phi Z_{1}.\]
Our first order of business in section 4 is to compute in generality the connecfion forms, torsion forms and Webster-Tanaka curvature
for the deformed structure.
Now we specialize the situation to $S^3$ and consider small deformations of the standand CR structure of the sphere.
In particular our goal is to consider the deformed structure on $S^3$ given by,
\[Z_{\bar{1}}^{t}=Z_{\bar{1}}^{\phi_t} = F(Z_{\bar{1}}+t\phi Z_{1}),\]
where $F=(1-t^{2}|\phi|^{2})^{-1/2}$, $Z_{\bar{1}}= \bar{z} _{2}\frac{\partial}{\partial Z_{1}}-\bar{z}_{1}\frac{\partial}{\partial Z_{2}}$
and $t\in (-\epsilon,\epsilon )$.
The factor $F$ is introduced to normalize the Levi form so that $h_{1\bar{1}}\equiv 1$. For this structure we compute the deformed Paneitz operator $P_{0}^{t}$.
The main goal in Section 4 is to study the variations of $P_{0}^{t}$.
We now consider the 3-sphere $S^{3}\subset\C^{2}\ni (z_{1},z_{2})$ and denote by
\[P_{p,q}=\textrm{span} \{ z_{1}^{a} z_{2}^{b} \bar{z}_{1}^{c} \bar{z}_{2}^{d}| a+b=p,\ c+d=q\} \]
and the spherical harmonics
\[H_{p,q}=\{f \in P_{p,q} |-\Delta_{s^{3}}f=(p+q)(p+q+2)f\}.\]
 For a given $\phi\in C^{\infty }(S^{3})$ one has the Fourier representation
\[\phi \sim \sum \phi _{pq}\]
where $\phi _{pq}$ is the projection of $\phi $ onto $H_{p,q}$.

\begin{de}
We say $\phi $ satisfies condition (BE) if and only if
\[\phi _{pq}\equiv 0\  \textrm{for}\ \  p<q+4,\ q=0,1,\cdots.\]
\end{de}

\begin{rem}
Since for $p>q$
\begin{equation*}
P_{p,q}=H_{p,q}\oplus\cdots\oplus H_{p-q,0}.
\end{equation*}
It follows that if $\phi\in P_{p,q}$, then $\phi$ satisfies (BE) if and only if $p\geq q+4$.
\end{rem}

Burns and Epstein proved in \cite{BE} that for $t\in (-\epsilon ,\epsilon )$ and $\phi $ satisfying (BE) the CR structure embeds into some $\C^{n}$.
Conversely Bland \cite{B} showed that embeddability of a CR structure close to the standard structure on $S^{3}$
implies condition (BE). To summarize we have \\

{\bf Theorem }[\ {\bf Burns-Epstein-Bland}].
A CR structure close to the standard structure on $S^3$ is embeddable if and only if $\phi$
satisfies condition (BE).\\

We define the space
 \[{\bf H}=C^{\infty}(S^3)\cap\left(\oplus_{p\geq 1}H_{p,0}\oplus H_{0,p}\right). \]
The main result proved in section 4 is

\begin{prop}\label{main3}
Let $\phi\in P_{p_{1},q_{1}}$. For a CR structure given by deformation by $\phi$ and close to the standard structure on $S^3$, i.e., $t\in(-\epsilon ,\epsilon )$,
the associated CR Paneitz operator is positive on ${\bf H}$, provided $\phi$ satisfies (BE).
\end{prop}

The proposition above is partial information about the behavior of
the CR Paneitz operator near the standard CR structure on $S^3$. We
will return to a fuller investigation in a later paper.

 We now outline the strategy of proof of Proposition
\ref{main3}. Since CR pluriharmonic functions are annihilated by the
Paneitz operator, we will study for $f\in {\bf H}$, the quadratic
form

\begin{equation}\label{qudrf}
I_{t}(f)=<P_{0}^{t}f, f>.
\end{equation}

We show in Theorem \ref{apfvf1} that for $f\in {\bf H}$ and $\phi\in C^{\infty}(S^3)$,

\begin{equation}\label{fvofqu}
\dot{I}_{t}(f)|_{t=0}=\frac{d}{dt}I_{t}(f)|_{t=0}\equiv 0.
\end{equation}
Thus the first variation all vanish. The (BE) condition does not
appear in the first variation formula, but appears in the second
variation formula. To compute the second variation, that is to
compute $\ddot{P}_{0}^{t}|_{t=0}$, we perform a Morse decomposition
of the functional $\ddot{I}_{t}$. It splits into a stable part,
called $D^{2}$ which we handle via Proposition \ref{dsp} and an
unstable part which is handled by proposition \ref{apsvf}. The (BE)
condition enters naturally into the unstable part by means of an
expressions $E$,
\[E=4\phi+i\phi_{0},\ \ \ \phi_{0}=T\phi.\]
In fact if $\phi$ satisfies (BE) then there is no unstable part.

Writing $f\in{\bf H}$ as
\[f=\sum_{k\geq 1}f^{k}+\sum_{k\geq 1}g^{k},\ \ f^{k}\in H_{k,0},\ \ g^{k}\in H_{0,k},\]
we get by throwing away the stable part:

\begin{prop}\label{main5}
For any $\phi\in C^{\infty}(S^3),\ f\in{\bf H}$,
\begin{equation*}
\begin{split}
\ddot{I}_{t}(f)|_{t=0}&=\frac{d^2}{dt^2}I_{t}(f)|_{t=0}\\
&\geq 2\sum_{k,l}\int_{S^3}(k|\phi|^{2}-E\bar{\phi})f^{k}_{1}\overline{f^{l}_{1}}+2\sum_{k,l}\int_{S^3}(k|\phi|^{2}-E\bar{\phi})g^{k}_{\bar 1}\overline{g^{l}_{\bar 1}}.
\end{split}
\end{equation*}
\end{prop}

We now invoke the Hopf fibration to perform the integration in the
right side of the theorem above. Viewing $S^3$ as a $S^1$ fibration
over $CP^{1}$ reduces the computation to doing Fourier series on
$(-\pi, \pi)$. This is the content of Proposition \ref{keyf}. This
proposition works for general $\phi$ if $k=l$, but we have been
unable to do the integration when $k\neq l$ unless $\phi\in
P_{p_{1},q_{1}}$. We get

\begin{prop}\label{main6}
For any $\phi\in P_{p_{1},q_{1}},\ f\in{\bf H}$,
\begin{equation*}
\ddot{I}_{t}(f)|_{t=0}
\geq 2\sum_{k}\int_{S^3}(k+p_{1}-q_{1}-4)|\phi|^{2}(|f^{k}_{1}|^{2}+|g^{k}_{\bar 1}|^{2}).
\end{equation*}
\end{prop}

We emphasize that the conclusion of Proposition \ref{main6} holds
for even those $\phi\in P_{p_{1},q_{1}}$ for which $p_{1}<q_{1}+4$,
i.e., for those values of $p_{1}, q_{1}$ that fail to satisfy
condition (BE).

If however $\phi$ satisfies (BE), then it is evident for $k\geq 1,\ \ k+p_{1}-q_{1}-4\geq 1$ and thus from the above theorem $\ddot{I}_{t}(f)|_{t=0}>0$,
for $f\in{\bf H}$. Combining this fact with (\ref{fvofqu}) and since $t\in(-\epsilon ,\epsilon )$ we see readily that if $\phi$ satisfies (BE), then for
$t\in(-\epsilon ,\epsilon )$, $f\in {\bf H}$, we have
\begin{equation}
I_{t}(f)=<P_{0}^{t}f,f>\ >0.
\end{equation}

We also point out some other results in section 4 of independent interest. One such result is Corollary \ref{cotofr},
which states if $\phi\in P_{p_{1},q_{1}}$, with $p_{1}=q_{1}+4$, then the deformed structure on $S^3$ also has zero torsion and conversely if $\phi$
is a homogeneous polynomial in $P_{p_{1},q_{1}}$, then for precisely those for which $p_{1}=q_{1}+4$, the new torsion will also vanish.

Lastly we comment that our second variation formula shows that for $\phi\in P_{p_{1},q_{1}}$ and $f\in H_{p,0}$ or $f\in H_{0,p}$, the possible negative
direction of $\ddot{I}_{t}(f)|_{t=0}$ can only lie in the space $f\in H_{p,0}$ or $f\in H_{0,p}$ for $p<q_{1}+4-p_{1}$. A more careful computation of
the second variation of Paneitz operator for Rossi's example $\phi\equiv 1$, which we have chosen not to display, shows that the negative directions are
given exactly by the functions $f=z_{1}, z_{2}, \bar{z}_{1}, \bar{z}_{2}$.

{\bf Acknowledgment.} The first author's research was supported in
part by NSF grant DMS-0855541, the second author's research was
supported in part by CIZE Foundation
 and in part by NSC 96-2115-M-008-017-MY3, and the third author's research was supported in part by
NSF grant DMS-0758601.

\section{The Embedding Criterion}\label{sec:1}
In this section, we will derive the Bochner formula for the Kohn Laplacian. We need some commutation relations, for which we refer the reader to Lee's paper \cite{L}.
This formula contains no term related to pseudohermitian torsion. In this sense it seems to be more natural than the one for the sublaplacian.
We have the following Bochner formula:
\begin{prop}\label{procombf}
For any complex-valued function $\varphi$ ,we have
\begin{equation}\label{combf}
\begin{split}
-\frac{1}{2} \Box _{b} | \bar{\partial} _{b} \varphi |^{2}&=
(\varphi _{\bar{1}\bar{1}}\bar{\varphi}_{11}+\varphi_{\bar{1}1}\bar{\varphi}_{1\bar{1}})\\
&-\frac{1}{2}<\bar{\partial}_{b}\varphi,\bar{\partial}_{b}\Box_{b}\varphi>
- <\bar{\partial}_{b}\Box_{b}\varphi,\bar{\partial}_{b}\varphi>\\
&-<\bar{P}\varphi,\bar{\partial}_{b}\varphi> + R | \bar{\partial}_{b}\varphi |^{2}\\
\end{split}
\end{equation}
\end{prop}

\begin{proof}
We calculate
\begin{equation}
\begin{split}
-\frac{1}{2}\Box_{b} | \bar{\partial}_{b}\varphi |^{2} & = -\frac{1}{2}\Box_{b}
< \varphi_{\bar{1}} \theta^{\bar{1}},\varphi_{\bar{1}} \theta^{\bar{1}}>\\
&=(\varphi_{\bar{1}}\bar{\varphi}_{1})_{\bar{1}1}\\
&=(\varphi_{\bar{1}\bar{1}}\bar{\varphi}_{1} + \varphi_{\bar{1}}\bar{\varphi}_{1\bar{1}})_{1}\\
&=\varphi_{\bar{1}\bar{1}1}\bar{\varphi}_{1} + \varphi_{\bar{1}\bar{1}}\bar{\varphi}_{11}
+\varphi_{\bar{1}1}\bar{\varphi}_{1\bar{1}} + \varphi_{\bar{1}}\bar{\varphi}_{1\bar{1}1},\\
&=\varphi_{\bar{1}\bar{1}}\bar{\varphi}_{11} + \varphi_{\bar{1}1}\bar{\varphi}_{1\bar{1}}
-\frac{1}{2}<\bar{\partial}_{b}\varphi , \bar{\partial}_{b}\Box_{b}\varphi>
+ \varphi_{\bar{1}\bar{1}1}\bar{\varphi}_{1},
\end{split}
\end{equation}
here, for the last equality, we use the identity
\begin{equation*}
-\frac{1}{2}<\bar{\partial}_{b}\varphi,\bar{\partial}_{b}\Box_{b}\varphi>=\varphi_{\bar{1}}\bar{\varphi}_{1\bar{1}1}
\end{equation*}

Therefore, the Bochner formula is completed if we show that
\begin{equation}\label{1}
\varphi_{\bar{1}\bar{1}1}\bar{\varphi}_{1}=-<\bar{\partial}_{b}\Box_{b}\varphi , \bar{\partial}_{b}\varphi>
-<\bar{P}\varphi,\bar{\partial}_{b}\varphi>+R | \bar{\partial}_{b}\varphi |^{2}.
\end{equation}
By the commutation relations, we have.
\begin{equation}\label{2}
\begin{split}
\varphi_{\bar{1}\bar{1}1} \bar{\varphi}_{1} & = (\varphi_{\bar{1}1\bar{1}}-i\varphi_{\bar{1}0}
+R\varphi_{\bar{1}})\bar{\varphi_{1}}\\
&=(\varphi_{1\bar{1}\bar{1}}-i\varphi_{0 \bar{1}}-i\varphi_{\bar{1}0}+R\varphi_{\bar{1}})\bar{\varphi}_{1}\\
&=(\bar{P}_{1}\varphi+iA_{\bar{1}\bar{1}}\varphi_{1}-i\varphi_{0\bar{1}}
-i\varphi_{\bar{1}0}+R\varphi_{\bar{1}})\bar{\varphi}_{1}\\
&=(\bar{P}_{1}\varphi)\bar{\varphi}_{1}+R\varphi_{\bar{1}}\bar{\varphi}_{1}
+(iA_{\bar{1}\bar{1}}\varphi-i\varphi_{0\bar{1}}-i\varphi_{\bar{1}0})\bar{\varphi}_{1}\\
&=<\bar{P}\varphi,\bar{\partial}_{b}\varphi>+R|\bar{\partial}_{b}\varphi|^{2}
+(iA_{\bar{1}\bar{1}}\varphi_{1}-i\varphi_{0\bar{1}}-i\varphi_{\bar{1}0})\bar{\varphi}_{1},\\
&=<\bar{P}\varphi,\bar{\partial}_{b}\varphi>+R|\bar{\partial}_{b}\varphi|^{2}
+2(iA_{\bar{1}\bar{1}}\varphi_{1}-i\varphi_{0\bar{1}})\bar{\varphi}_{1},
\end{split}
\end{equation}
and
\begin{equation}\label{3}
\begin{split}
&-\frac{1}{2}<\bar{\partial}_{b}\Box_{b}\varphi,\bar{\partial}_{b}\varphi>\\
&=<\varphi_{\bar{1}1\bar{1}}\theta^{\bar{1}},\varphi_{\bar{1}}\theta^{\bar{1}}>\\
&=\varphi_{\bar{1}1\bar{1}}\bar{\varphi}_{1}\\
&=(\varphi_{1\bar{1}\bar{1}}-i \varphi_{0\bar{1}})\bar{\varphi}_{1}\\
&=(\bar{P}_{1}\varphi+i A_{\bar{1}\bar{1}}\varphi_{1}-i\varphi_{0\bar{1}})\bar{\varphi}_{1}\\
&=<\bar{P}\varphi,\bar{\partial}_{b}\varphi>+(iA_{\bar{1}\bar{1}}\varphi_{1}-i\varphi_{0\bar{1}})\bar{\varphi}_{1}
\end{split}
\end{equation}
Combining (\ref{2}) and (\ref{3}) ,we obtain (\ref{1}). This completes the Proposition.
\end{proof}

We now prove Theorem 1.3.

\textbf{Proof of Theorem 1.3 :} Let $\varphi$ be an eigenfunction with respect to a nonzero eigenvalue $\lambda$, that is, $\varphi$ is not a CR function.
Taking the integral of both sides of the Bochner formula (\ref{combf}), we have
\begin{equation}
\begin{split}
0&=\int\varphi_{\bar{1}\bar{1}}\bar{\varphi}_{11}+\int\varphi_{\bar{1}1}\bar{\varphi}_{1\bar 1}-\frac{3}{2}\lambda\int|\bar{\partial}_{b}\varphi|^{2}\\
&+\int <P_{0}\varphi,\varphi>+\int R|\bar{\partial}_{b}\varphi|^{2}.
\end{split}
\end{equation}
On the other hand,
\[\int\varphi_{\bar{1}1}\bar{\varphi}_{1\bar 1}=\frac{1}{4}\int<\Box_{b}\varphi,\Box_{b}\varphi>=\frac{\lambda}{2}\int |\bar{\partial}_{b}\varphi|^{2}.\]
Taking together the above two formulae, we obtain

\begin{equation}
\begin{split}
\lambda\int |\bar{\partial}_{b}\varphi|^{2}&=\int|\varphi_{\bar{1}\bar{1}}|^{2}+\int <P_{0}\varphi,\varphi>+\int R|\bar{\partial}_{b}\varphi|^{2}\\
&\geq\int <P_{0}\varphi,\varphi>+\int R|\bar{\partial}_{b}\varphi|^{2}.
\end{split}
\end{equation}
Therefore, if $P_{0}$ is nonnegative and $R>0$, then we immediately
have that $\lambda\geq\min{R}$. Since the spectrum
$\textrm{spec}(\Box_{b})$ of the Kohn Laplacian in $(0,\infty)$ only
consists of point eigenvalues (see Theorem 1.3 in \cite{BE}), it
follows that the range of $\Box_{b}$ is closed. Applying the result
of Kohn \cite{K}, we conclude $M$ is embeddable.

To prove the second part of Theorem \ref{main1} note if the CR
Yamabe constant $>0$, then we can choose a contact form such that
the Webster curvature with respect to this contact form is positive
and so we conclude by the first part of Theorem \ref{main1}.

\begin{rem}
The estimate for the nonzero eigenvalues is sharp. For example, the standard sphere $S^{3}$ as a pseudohermitian 3-manifold has the smallest
nonzero eigenvalue $\lambda=2=R$, for details, see \cite{C}.
\end{rem}

\begin{rem}
In general, let $M^{2n+1}$ be a pseudohermitian manifold. The Bochner formula for the Kohn Laplacian is as follows:
\begin{equation}\label{bfh}
\begin{split}
-\frac{1}{2} \Box _{b} | \bar{\partial} _{b} \varphi |^{2}&=
\sum_{\alpha, \beta }(\varphi _{\bar{\alpha }\bar{\beta }}\bar{\varphi}_{\alpha \beta }+\varphi_{\bar{\alpha }\beta }\bar{\varphi}_{\alpha \bar{\beta }})\\
&-\frac{1}{2n}<\bar{\partial}_{b}\varphi,\bar{\partial}_{b}\Box_{b}\varphi>
- \frac{(n+1)}{2n}<\bar{\partial}_{b}\Box_{b}\varphi,\bar{\partial}_{b}\varphi>\\
&-\frac{1}{n}<\bar{P}\varphi,\bar{\partial}_{b}\varphi> +\frac{(n-1)}{n}<P\bar{\varphi}, \partial_{b}\bar{\varphi}>\\
&+ Ric(\nabla_{b}\varphi_{\C}, \nabla_{b}\varphi_{\C}),
\end{split}
\end{equation}
where $\nabla_{b}\varphi_{\C}$ is the corresponding complex $(1,0)$-vector of $\nabla_{b}\varphi$. The proof of (\ref{bfh}) is the same as (\ref{combf}).
Again, in case $n=2$, using this formula we also obtain that the sharp lower bound of nonzero eigenvalues of the Kohn Laplacian $\Box_{b}$ is $\frac{4}{3}k_{0}$,
provided that the Ricci curvature has the lower bound:
\[Ric(X,X)\geq k_{0}|X|^{2}, \]
for some $k_{0}$ and for all complex $(1,0)$-vector X. Unfortunately, in the higher dimensional cases $n\geq 3$, the coefficient $\frac{(n-1)}{n}$ of the term
$<P\bar{\varphi}, \partial_{b}\bar{\varphi}>$ is too large to get the lower bound of nonzero eigenvalues of the Kohn Laplacian $\Box_{b}$ immediately.
\end{rem}

\begin{exa}\label{ex2-4}
In this example, we shall compute the Webster curvature of Rossi's global nonembeddability example together with a suitable contact structure and
show that the associated CR Paneitz operators are not nonnegative.
Let $S^{3}=\{(z_1,z_2)\in\C^{2}| |z_1|^{2}+|z_2|^{2}-1=0\}$
be the boundary of the unit ball in $\C^{2}$ with the induced CR structure given by the complex vector field
\[Z_{1}=\bar{z}_{2}\frac{\partial}{\partial z_{1}}-\bar{z}_{1}\frac{\partial}{\partial z_{2}},\]
and contact form
\[\theta=\frac{i(\bar{\partial}u-\partial u)}{2}|_{S^{3}},\]
where $u=|z_1|^{2}+|z_2|^{2}-1$. Taking the admissible coframe
\[\theta^{1}=z_{2}dz_{1}-z_{1}dz_{2},\]
we have $d\theta=i\theta^{1}\wedge\theta^{\bar 1}$. Rossi's example is the CR manifold $S^{3}$ together with the CR structure
given by
\[L_{t}=Z_{1}+t\bar{Z}_{1}, \]
for all $t\in\R$ and $t\neq 1, -1$. Now, for $|t|<1$, taking the contact form
\[\theta(t)=\theta,\]
and the admissible coframe
\[\theta^{1}(t)=\frac{1}{\sqrt{1-t^{2}}}(\theta^{1}-t\theta^{\bar 1}),\]
we have $d\theta(t)=i\theta^{1}(t)\wedge\theta^{\bar 1}(t)$ and the following Proposition:

\begin{prop}
For $|t|<1$, with respect to the coframe $\{\theta(t), \theta^{1}(t), \theta^{\bar 1}(t)\}$, the connection form and pseudohermitian torsion are as follows
\begin{equation}
\theta_{1}{}^{1}(t)=\theta_{1}{}^{1}-\frac{4t^{2}i}{1-t^{2}}\theta=\frac{-2(1+t^{2})}{1-t^{2}}i\theta;
\end{equation}
and
\[\tau^{1}(t)=\frac{4ti}{1-t^2}\theta^{\bar 1}(t),\]
where
\[\theta_{1}{}^{1}=-\bar{z}_{1}dz_{1}-\bar{z}_{2}dz_{2}+z_{1}d\bar{z}_{1}+z_{2}d\bar{z}_{2}=-2i\theta.\]
In addition, the Webster curvature is $R(t)=\frac{2(1+t^{2})}{1-t^2}$.
\end{prop}
\begin{proof}
We just check that forms $\theta_{1}{}^{1}(t),\ \tau^1(t)$ satisfy the equations (\ref{id10}).
Finally, after a direct computation, we see that $d\theta_{1}{}^{1}(t)=\frac{2(1+t^{2})}{1-t^2}\theta^1\wedge\theta^{\bar 1}$, so $R(t)=\frac{2(1+t^{2})}{1-t^2}$.
\end{proof}
Similarly, for $|t|>1$, take the contact form $\theta(t )$ and the admissible coframe $\theta^{1}(t)$ as follows:
\[\theta(t)=-\theta,\ \ \ \ \ \theta^{1}(t)=\frac{1}{\sqrt{t^{2}-1}}(\theta^{1}+t\theta^{\bar 1}).\]
Then we have
\[\tau^{1}(t)=\frac{4ti}{1-t^2}\theta^{\bar 1}(t)\ \textrm{and}\ R(t)=\frac{2(1+t^{2})}{t^{2}-1}\]
From Theorem \ref{main1} and the above Proposition, we immediately obtain that the CR Paneitz operator of Rossi's nonembeddable manifolds are negative.

\end{exa}

\section{Stability of Embeddability}

We now consider stability issues, see \cite{BE} and \cite{Le} for earlier work. We have a fixed CR structure on a compact manifold $(M^{3}, \theta, J)$.
Let us denote by $\bar{L}$ the CR vector field on $M^{3}$. We now perturb $\bar{L}$ by a smooth family of functions
$\varphi(\cdot, t)=\varphi_{t}(\cdot)$ ,where $(\cdot)$ represents a point on M, and $t\in(-\varepsilon ,\varepsilon )$.
We assume always,

\begin{equation}\label{se1}
D^{\alpha}_{z,s}\varphi (z,s,t)|_{t=0}=0,\ \ |\alpha|\leq l_{0},\ \ l_{0}\geq 4,\ \ (z,s)\in M.
\end{equation}

We define

\begin{equation}\label{se2}
\bar{L}_{t}=\bar{L}+\varphi(\cdot, t)L.
\end{equation}

Associated to $\bar{L}_{t}$ we form the associated $\bar{\partial}_{b}^{(t)}$-Laplacian operator,
\begin{equation}\label{se3}
\Box_{b}^{(t)}=\bar{\partial}_{b}^{(t)*}\bar{\partial}_{b}^{(t)}.
\end{equation}

We now use our main result to guarantee embedding of our CR structure in $\C^{N}$. Thus we assume that along the deformation
path in $t$,
\begin{equation}\label{se4}
\textrm{the associated Paneitz operator}\ P_{0}^{(t)}\geq 0,\ \ \ \ \ \ \ \ \ \ \ \ \ \ \ \ \ \ \ \ \ \ \ \ \ \ \ \ \ \ \ \ \ \ \ \ \ \ \ \ \ \ \ \ \ \ \ \ \ \ \ \ \ \ \ \ \ \ \ \
\end{equation}
\begin{equation}\label{se5}
\textrm{the CR Yamabe constant}\geq c>0.\ \ \ \ \ \ \ \ \ \ \ \ \ \ \ \ \ \ \ \ \ \ \ \ \ \ \ \ \ \ \ \ \ \ \ \ \ \ \ \ \ \ \ \ \ \ \ \ \ \ \ \ \ \ \ \ \ \ \ \ \ \ \ \ \ \
\end{equation}
By our main result (Theorem \ref{main1}),
using (\ref{se4}) and (\ref{se5}) it follows that
\begin{equation}\label{se6}
\lambda_{1}(\Box_{b}^{(t)})\geq\nu >0,
\end{equation}
with $\nu$ independent of $t$. Thus by using the construction of Boutet de Monvel in \cite{BD} or the exposition in Chen and Shaw's book
\cite{CS}, we can embed for $\varepsilon >0$ small enough the CR structures via a map $\Psi_{t} $ into the same $\C^{N}$, i.e.,
\begin{equation}\label{se7}
\Psi_{t}: (M, \theta, J_{t})\longrightarrow \C^{N}.
\end{equation}
The question arises if the maps $\Psi_{t}$ are close in say the sup-norm in $t$. We have the following theorem which we re-state from the introduction:

THEOREM 1.5.
\textit{\ Under} (\ref{se1}), (\ref{se4}), (\ref{se5}), \textit{\ for any} $\delta>0$, \textit{\ there exists} $\varepsilon >0$, \textit{\ so that}
\[\sup_{t\in[-\varepsilon, \varepsilon]}\|\Psi_{t}-\Psi_{0}\|_{C^{k}(M)}< \delta,\ \ k=k(l_{0}).\]

\begin{proof}
The proof of this theorem is abstract and relies on an identity in \cite{BE}. We use Proposition 5.55 in \cite{BE}.
We denote the projection into the zero eigenspace of $\Box_{b}^{(t)}$ by $\Im^{\varphi_{t}} $, which is the Szego projector.
By the spectral theorem and (\ref{se6}) if $|\lambda|=\nu/2$, the resolvent $(\Box_{b}^{(t)}-\lambda)^{-1}$ is well-defined and
so
\[\Im^{\varphi_{t}}=\int\limits_{|\lambda |=\nu/2}(\Box_{b}^{(t)}-\lambda)^{-1}d\lambda,\]
and it is immediate that $\Im^{\varphi_{t}}$ is a bounded operator on $L^{2}(M)$. As observed in \cite{BE} as a consequence of
the above fact and their identity (5.58) they obtain the inequality (5.60) which we re-state,
\begin{equation}\label{se8}
\|\Im^{\varphi_{t}}-\Im^{\varphi_{0}}\|_{L^{2}(M)}\leq C\nu A\|\varphi_{t}-\varphi_{0}\|_{L^{\infty}(M)},
\end{equation}
where $A$ is the sup norm of some high enough derivative of $\varphi_{t}-\varphi_{0}$. But by our hypothesis (\ref{se1})
the right side of (\ref{se8}) is smaller than $\delta>0$, for $\varepsilon>0 $, sufficiently small.\\
Now recall the construction of Boutet de Monvel. Using the notation in \cite{CS}, page 318, the embedding for each coordinate
chart is given by a CR function $h_{t}$ (we are in CR dimension $1$), where
\begin{equation}\label{se9}
h_{t}=\Im^{\varphi_{t}}\big(\psi e^{-\tau\varphi_{p}}\big),\ \ \tau\rightarrow  \infty.
\end{equation}
Now note $h_{t}-h_{0}$ also satisfy an equation, that is,
\begin{equation}\label{se10}
\Box_{b}^{(t)}(h_{t}-h_{0})=\big(\Box_{b}^{(0)}-\Box_{b}^{(t)}\big)(h_{0}).
\end{equation}
From (\ref{se8}), (\ref{se9}),\ $\|h_{t}-h_{0}\|_{L^{2}(M)}<\delta$.\\
From (\ref{se1}) and (\ref{se10}), the right side of (\ref{se10}) is small in the $C^{\infty}$-norm. Since we have (\ref{se6}),
it now implies by sub-elliptic regularity that for $\delta>0$, there exists $\varepsilon_0 $,
\begin{equation}\label{se11}
\sup_{t\in(-\varepsilon_{0},\varepsilon_{0})}\|h_{t}-h_{0}\|_{C^{\infty}(M)}\leq\delta.
\end{equation}
In fact by differentiation of (\ref{se10}) in $t$, we may also obtain higher stability in $t$, provided we replace (\ref{se1})
by the stronger hypothesis that $D_{z,s}^{\alpha}D_{t}^{\beta}\varphi(z,s,0)=0$ for large enough $|\alpha|, |\beta|$. This
proves our theorem since on coordinate charts of $M$, the map $\Psi_{t}$ is given by $h_{t}$.
\end{proof}

\section{The second variation of the Paneitz operator}
Our goal in this section is to investigate CR structures close to the standard structure on $S^3$ and prove Theorem \ref{main3}, which is a converse to Theorem \ref{main1}.
To achieve our goal we compute the second variation of the Paneitz operator. Let $(M, J, \theta)$ be a three-dimensional pseudo-hermitian manifold.
In the computation, the contact form $\theta$ is always fixed and we suppose that the CR structure $J$ is given by the the $(0,1)$-complex vector field
$Z_{\bar 1}$.

Suppose $\phi\in C^{\infty}(M)$ with $|\phi|<1$. Then the complex vector field
\begin{equation}\label{dfcr}
Z_{\bar 1}{}^{\phi}=Z_{\bar 1}+\phi Z_{1}
\end{equation}
defines a strictly pseudoconvex CR structure on $M$.

For the purpose of computing the second variation, we need to know exactly what the connection and torsion forms are for the manifold with CR
structure defined by the complex vector field (\ref{dfcr}). Therefore, first of all, we focus on the computation of the connection form and torsion
form and then use them to obtain the second variation of the Paneitz operator.

Let $\theta^{1}$ denote the $(1,0)$-form dual to $Z_{1}$. We take
\begin{equation}\label{admicofr}
\theta^{1}{}_{\phi}=F(\phi)(\theta^{1}-\phi\theta^{\bar 1})
\end{equation}
as an admissible coframe, where
\begin{equation*}
F=F(\phi)=\frac{1}{(1-|\phi|^{2})^{1/2}},
\end{equation*}
which is a real function. For simplifying the computation, we normalize $Z_{\bar 1}{}^{\phi}$ by setting
\begin{equation*}
Z_{\bar 1}{}^{\phi}=F(\phi)(Z_{\bar 1}+\phi Z_{1})
\end{equation*}
such that $\{Z_{1}{}^{\phi}, Z_{\bar 1}{}^{\phi}, T\}$ is dual to $\{\theta^{1}{}_{\phi}, \theta^{\bar 1}{}_{\phi}, \theta\}$ and $h_{1\bar 1}{}^{\phi}\equiv h_{1\bar 1}$.
Now we are ready to compute the connection and torsion forms, which are denoted by $\theta_{1}{}^{1}{}_{\phi}$ and $\tau^{1}{}_{\phi}$, respectively. They are
determined by the following structure equations:
\begin{equation}\label{strueqfordef}
\begin{split}
d\theta^{1}{}_{\phi}&=\theta^{1}{}_{\phi}\wedge\theta_{1}{}^{1}{}_{\phi} +\theta\wedge\tau^{1}{}_{\phi}\\
\tau^{1}{}_{\phi}&=0,\ \ \ \textrm{mod}\ \theta^{\bar 1}{}_{\phi}\\
h^{1\bar 1}{}_{\phi}dh_{1\bar 1}{}^{\phi}&=\theta_{1}{}^{1}{}_{\phi}+\theta_{\bar 1}{}^{\bar 1}{}_{\phi},
\end{split}
\end{equation}
where $h^{1\bar 1}{}_{\phi}$ is the inverse of $h_{1\bar 1}{}^{\phi}$. Denote $\tau^{1}{}_{\phi}=A^{1}{}_{\bar 1}{}^{\phi}\theta^{\bar 1}{}_{\phi}$. Then we
have the following proposition

\begin{prop}\label{expofcoto}
We have
\begin{equation}
\begin{split}
\theta_{1}{}^{1}{}_{\phi}&=\theta_{1}{}^{1}-F^{-1}dF-F^{-1}(B_{11}\theta^{1}+B_{12}\theta^{\bar 1}+B_{13}\theta);\\
A^{1}{}_{\bar 1}{}^{\phi}&=A^{1}{}_{\bar 1}-F^{2}\Big(\phi_{0}+\phi\theta_{1}{}^{1}(T)-\phi\theta_{\bar 1}{}^{\bar 1}(T)+\phi^{2}A^{\bar 1}{}_{1}-|\phi|^{2}A^{1}{}_{\bar 1}\Big),
\end{split}
\end{equation}
where
\begin{equation}
\begin{split}
B_{11}&=F^{2}\Big(-2F_{1}-\bar{\phi}F\theta_{\bar 1}{}^{\bar 1}(Z_{\bar 1})+\bar{\phi}F\theta_{1}{}^{1}(Z_{\bar 1})\\
&-F\bar{\phi}_{\bar 1}-\bar{\phi}F\phi_{1}-2\bar{\phi}F_{\bar 1}-|\phi|^{2}F\theta_{1}{}^{1}(Z_{1})+|\phi|^{2}F\theta_{\bar 1}{}^{\bar 1}(Z_{1})\Big);\\
B_{12}&=F^{2}\Big(2|\phi|^{2}F_{\bar 1}+\phi F\theta_{1}{}^{1}(Z_{1})-\phi F\theta_{\bar 1}{}^{\bar 1}(Z_{1})\\
&+F\phi_{1}+\phi F\bar{\phi}_{\bar 1}+2\phi F_{1}+|\phi|^{2}F\theta_{\bar 1}{}^{\bar 1}(Z_{\bar 1})-|\phi|^{2}F\theta_{1}{}^{1}(Z_{\bar 1})\Big);\\
B_{13}&=F^{3}\Big(-\bar{\phi}(\phi_{0}+\phi\theta_{1}{}^{1}(T)-\phi\theta_{\bar 1}{}^{\bar 1}(T))+\bar{\phi}A^{1}{}_{\bar 1}-\phi A^{\bar 1}{}_{1}\Big).
\end{split}
\end{equation}
\end{prop}
\begin{proof}
From the structure equations (\ref{id10}), we have
\begin{equation*}
d\theta^{1}=\theta^{1}\wedge\theta_{1}{}^{1}+\theta\wedge\tau^{1},
\end{equation*}
thus
\begin{equation}\label{coto1}
\begin{split}
d\theta^{1}{}_{\phi}&=d(F\theta^{1}-\phi F\theta^{\bar 1})\\
&=dF\wedge\theta^{1}+Fd\theta^{1}-d(\phi F)\wedge\theta^{\bar 1}-\phi Fd\theta^{\bar 1}\\
&=dF\wedge\theta^{1}+F\theta^{1}\wedge\theta_{1}{}^{1}+F\theta\wedge\tau^{1}\\
&-d(\phi F)\wedge\theta^{\bar 1}-\phi F\theta^{\bar 1}\wedge\theta_{\bar 1}{}^{\bar 1}-\phi F\theta\wedge\tau^{\bar 1}\\
&=(dF-F\theta_{1}{}^{1})\wedge\theta^{1}+(d(\phi F)-\phi F\theta_{\bar 1}{}^{\bar 1})\wedge\theta^{\bar 1}+(-F\tau^{1}+\phi F\tau^{\bar 1})\wedge\theta.
\end{split}
\end{equation}
On the other hand,
\begin{equation}\label{coto2}
\begin{split}
&\theta^{1}{}_{\phi}\wedge\theta_{1}{}^{1}{}_{\phi}+\theta\wedge\tau^{1}{}_{\phi}\\
=&(F\theta^{1}-\phi F\theta^{\bar 1})\wedge\theta_{1}{}^{1}{}_{\phi}+\theta\wedge\tau^{1}{}_{\phi}\\
=&(-F\theta_{1}{}^{1}{}_{\phi})\wedge\theta^{1}+(\phi F\theta_{1}{}^{1}{}_{\phi})\wedge\theta^{\bar 1}+(-\tau^{1}{}_{\phi})\wedge\theta.
\end{split}
\end{equation}
Comparing (\ref{coto1}) and (\ref{coto2}), we see, by Cartan's lemma, that there exists complex-valued functions $B_{ij}, \ i,j=1, 2, 3$ such that
\begin{equation}\label{coto3}
\begin{split}
-F\theta_{1}{}^{1}{}_{\phi}&=(dF-F\theta_{1}{}^{1})+B_{11}\theta^{1}+B_{12}\theta^{\bar 1}+B_{13}\theta;\\
\phi F\theta_{1}{}^{1}{}_{\phi}&=-d(\phi F)+\phi F\theta_{\bar 1}{}^{\bar 1})+B_{21}\theta^{1}+B_{22}\theta^{\bar 1}+B_{23}\theta;\\
-\tau^{1}{}_{\phi}&=-F\tau^{1}+\phi F\tau^{\bar 1}+B_{31}\theta^{1}+B_{32}\theta^{\bar 1}+B_{33}\theta,
\end{split}
\end{equation}
and
\begin{equation}\label{coto4}
B_{ij}=B_{ji}.
\end{equation}
Therefore, from (\ref{coto3}), we have
\begin{equation*}
\begin{split}
&A^{1}{}_{\bar 1}{}^{\phi}(F\theta^{\bar 1}-\bar{\phi}F\theta^{1})\\=&A^{1}{}_{\bar 1}{}^{\phi}\theta^{\bar 1}\\
=&\tau^{1}{}_{\phi}\\
=&F\tau^{1}-\phi F\tau^{\bar 1}-B_{31}\theta^{1}-B_{32}\theta^{\bar 1}-B_{33}\theta\\
=&FA^{1}{}_{\bar 1}\theta^{\bar 1}-\phi FA^{\bar 1}{}_{1}\theta^{1}-B_{31}\theta^{1}-B_{32}\theta^{\bar 1}-B_{33}\theta,
\end{split}
\end{equation*}
hence, comparing the coeficients, we get

\begin{equation}\label{coto6}
FA^{1}{}_{\bar 1}{}^{\phi}=FA^{1}{}_{\bar 1}-B_{32};
\end{equation}

\begin{equation}\label{coto7}
-F\bar{\phi}A^{1}{}_{\bar 1}{}^{\phi}=-F\phi A^{\bar 1}{}_{1}-B_{31};
\end{equation}

\begin{equation}\label{coto8}
B_{33}=0
\end{equation}
Taken together (\ref{coto6}) and (\ref{coto7}) implies
\begin{equation*}
A^{1}{}_{\bar 1}{}^{\phi}=A^{1}{}_{\bar 1}-\frac{B_{32}}{F}=\frac{\phi}{\bar{\phi}}A^{\bar 1}{}_{1}+\frac{B_{31}}{\bar{\phi}F},
\end{equation*}
that is,
\begin{equation}\label{coto9}
\bar{\phi}FA^{1}{}_{\bar 1}-\bar{\phi}B_{32}=\phi FA^{\bar 1}{}_{1}+B_{31}.
\end{equation}
Now, from (\ref{coto3}) again, multiplying the first formula by $\phi$ and adding the second formula, we get
\begin{equation*}
\begin{split}
0&=\phi(dF-F\theta_{1}{}^{1})+\phi B_{11}\theta^{1}+\phi B_{12}\theta^{\bar 1}+\phi B_{13}\theta\\
&-d(\phi F)+\phi F\theta_{\bar 1}{}^{\bar 1}+B_{21}\theta^{1}+B_{22}\theta^{\bar 1}+B_{23}\theta\\
&=-Fd\phi-\phi F\theta_{1}{}^{1}+\phi F\theta_{\bar 1}{}^{\bar 1}+(\phi B_{11}+B_{21})\theta^{1}+(\phi B_{12}+B_{22})\theta^{\bar 1}
+(\phi B_{13}+B_{23})\theta\\
&=(-F\phi_{1}-\phi F\theta_{1}{}^{1}(Z_{1})+\phi F\theta_{\bar 1}{}^{\bar 1}(Z_{1})+\phi B_{11}+B_{21})\theta^{1}\\
&+(-F\phi_{\bar 1}-\phi F\theta_{1}{}^{1}(Z_{\bar 1})+\phi F\theta_{\bar 1}{}^{\bar 1}(Z_{\bar 1})+\phi B_{12}+B_{22})\theta^{\bar 1}\\
&+(-F\phi_{0}-\phi F\theta_{1}{}^{1}(T)+\phi F\theta_{\bar 1}{}^{\bar 1}(T)+\phi B_{13}+B_{23})\theta,
\end{split}
\end{equation*}
that is,
\begin{equation}\label{coto10}
-F\phi_{1}-\phi F\theta_{1}{}^{1}(Z_{1})+\phi F\theta_{\bar 1}{}^{\bar 1}(Z_{1})+\phi B_{11}+B_{21}=0
\end{equation}
\begin{equation}\label{coto11}
-F\phi_{\bar 1}-\phi F\theta_{1}{}^{1}(Z_{\bar 1})+\phi F\theta_{\bar 1}{}^{\bar 1}(Z_{\bar 1})+\phi B_{12}+B_{22}=0
\end{equation}
\begin{equation}\label{coto12}
-F\phi_{0}-\phi F\theta_{1}{}^{1}(T)+\phi F\theta_{\bar 1}{}^{\bar 1}(T)+\phi B_{13}+B_{23}=0.
\end{equation}
Multiplying (\ref{coto9}) by $\phi$ and subtracting (\ref{coto12}) we obtain
\begin{equation}\label{coto14}
\begin{split}
B_{23}&=\frac{F\phi_{0}+\phi F\theta_{1}{}^{1}(T)-\phi F\theta_{\bar 1}{}^{\bar 1}(T)+\phi^{2}FA^{\bar 1}{}_{1}-|\phi|^{2}FA^{1}{}_{\bar 1}}{(1-|\phi|^{2})}\\
&=F^{3}(\phi_{0}+\phi\theta_{1}{}^{1}(T)-\phi\theta_{\bar 1}{}^{\bar 1}(T)+\phi^{2}A^{\bar 1}{}_{1}-|\phi|^{2}A^{1}{}_{\bar 1}).
\end{split}
\end{equation}
Since $F^{2}=1+|\phi|^{2}F^{2}$, substituting (\ref{coto14}) into (\ref{coto9}), we obtain
\begin{equation}\label{coto15}
B_{13}=F^{3}\Big(-\bar{\phi}(\phi_{0}+\phi\theta_{1}{}^{1}(T)-\phi\theta_{\bar 1}{}^{\bar 1}(T))+\bar{\phi}A^{1}{}_{\bar 1}-\phi A^{\bar 1}{}_{1}\Big)
\end{equation}
Now, substituting (\ref{coto14}) into (\ref{coto6}), we obtain
\begin{equation}\label{coto16}
A^{1}{}_{\bar 1}{}^{\phi}=A^{1}{}_{\bar 1}-F^{2}\Big(\phi_{0}+\phi\theta_{1}{}^{1}(T)-\phi\theta_{\bar 1}{}^{\bar 1}(T)+\phi^{2}A^{\bar 1}{}_{1}-|\phi|^{2}A^{1}{}_{\bar 1}\Big)
\end{equation}
Finally, to complete the proof of the proposition, we need to determine $B_{11}$ and $B_{12}$. Taking the conjugate of the first formula of (\ref{coto3}),
we get
\begin{equation*}
\begin{split}
-F\theta_{\bar 1}{}^{\bar 1}{}_{\phi}&=(dF-F\theta_{\bar 1}{}^{\bar 1})+\overline{B_{11}}\theta^{\bar 1}+\overline{B_{12}}\theta^{1}+\overline{B_{13}}\theta\\
&=(dF-F(-\theta_{1}{}^{1}+h^{1\bar 1}dh_{1\bar 1}))+\overline{B_{11}}\theta^{\bar 1}+\overline{B_{12}}\theta^{1}+\overline{B_{13}}\theta\\
&=dF+F\theta_{1}{}^{1}-Fh^{1\bar 1}dh_{1\bar 1}+\overline{B_{11}}\theta^{\bar 1}+\overline{B_{12}}\theta^{1}+\overline{B_{13}}\theta.
\end{split}
\end{equation*}
On the other hand,
\begin{equation*}
\begin{split}
-F\theta_{\bar 1}{}^{\bar 1}{}_{\phi}&=-F(-\theta_{1}{}^{1}{}_{\phi}+h^{1\bar 1}dh_{1\bar 1})\\
&=-(dF-F\theta_{1}{}^{1})-B_{11}\theta^{1}-B_{12}\theta^{\bar 1}-B_{13}\theta-Fh^{1\bar 1}dh_{1\bar 1},
\end{split}
\end{equation*}
where the last equality is due to the first formula of (\ref{coto3}).
Comparing the above two formula, we have
\begin{equation*}
2dF=(-B_{11}-\overline{B_{12}})\theta^{1}+(-B_{12}-\overline{B_{11}})\theta^{\bar 1}+(-B_{13}-\overline{B_{13}})\theta,
\end{equation*}
which implies that
\begin{equation}\label{coto17}
B_{12}=-\overline{B_{11}}-2F_{\bar 1}.
\end{equation}
Substituting (\ref{coto17}) into (\ref{coto10}), we get
\begin{equation}\label{coto18}
\phi B_{11}=\overline{B_{11}}+2F_{\bar 1}+F\phi_{1}+\phi F\theta_{1}{}^{1}(Z_{1})-\phi F\theta_{\bar 1}{}^{\bar 1}(Z_{1}).
\end{equation}

Now multiplying (\ref{coto18}) by $\bar{\phi}$ and subtracting the conjugate of (\ref{coto18}), we obtain
\begin{equation}\label{coto21}
\begin{split}
B_{11}&=F^{2}\Big(-2F_{1}-\bar{\phi}F\theta_{\bar 1}{}^{\bar 1}(Z_{\bar 1})+\bar{\phi}F\theta_{1}{}^{1}(Z_{\bar 1})\\
&-F\bar{\phi}_{\bar 1}-\bar{\phi}F\phi_{1}-2\bar{\phi}F_{\bar 1}-|\phi|^{2}F\theta_{1}{}^{1}(Z_{1})+|\phi|^{2}F\theta_{\bar 1}{}^{\bar 1}(Z_{1})\Big).
\end{split}
\end{equation}
Substituting this into (\ref{coto17}), we get
\begin{equation}\label{coto22}
\begin{split}
B_{12}&=F^{2}\Big(2|\phi|^{2}F_{\bar 1}+\phi F\theta_{1}{}^{1}(Z_{1})-\phi F\theta_{\bar 1}{}^{\bar 1}(Z_{1})\\
&+F\phi_{1}+\phi F\bar{\phi}_{\bar 1}+2\phi F_{1}+|\phi|^{2}F\theta_{\bar 1}{}^{\bar 1}(Z_{\bar 1})-|\phi|^{2}F\theta_{1}{}^{1}(Z_{\bar 1})\Big).
\end{split}
\end{equation}
This finishes the proof of the proposition.
\end{proof}

According to Example \ref{ex2-4}, we see that if $S^{3}$ is the 3-sphere with the standard CR structure and contact form then
\begin{equation*}
A^{1}{}_{\bar 1}\equiv 0,\ \ h_{1\bar 1}\equiv 1,\ \ R\equiv 2,\ \ \theta_{1}{}^{1}(T)=-2i,\ \ \theta_{1}{}^{1}(Z_{1})=\theta_{1}{}^{1}(Z_{\bar 1})=0.
\end{equation*}
Therefore we have the following corollary.

\begin{cor}
On $S^{3}$, the connection form and torsion with respect to the CR structure given by
\begin{equation*}
Z_{\bar 1}{}^{\phi}=Z_{\bar 1}+\phi Z_{1}
\end{equation*}
are
\begin{equation}\label{coto23}
\begin{split}
\theta_{1}{}^{1}{}_{\phi}&=\theta_{1}{}^{1}-F^{-1}dF-F^{-1}(B_{11}\theta^{1}+B_{12}\theta^{\bar 1}+B_{13}\theta);\\
A^{1}{}_{\bar 1}{}^{\phi}&=-F^{2}(\phi_{0}-4i\phi),
\end{split}
\end{equation}
where
\begin{equation}\label{coto24}
\begin{split}
B_{11}&=F^{2}(-2F_{1}-F\bar{\phi}_{\bar 1}-\bar{\phi}F\phi_{1}-2\bar{\phi}F_{\bar 1});\\
B_{12}&=F^{2}(2|\phi|^{2}F_{\bar 1}+F\phi_{1}+\phi F\bar{\phi}_{\bar 1}+2\phi F_{1});\\
B_{13}&=-\bar{\phi}F^{3}(\phi_{0}-4i\phi).
\end{split}
\end{equation}
\end{cor}

\begin{cor}\label{cotofr}
We have that on $S^3$
\begin{equation}
A^{1}{}_{\bar 1}{}^{\phi}=0\Longleftrightarrow  \phi_{0}=4i\phi \Longleftrightarrow  \phi\in P_{p,q},\ p=q+4,
\end{equation}
where
\begin{equation}
P_{p,q}=sp\{z_{1}^{a}z_{2}^{b}\bar{z_{1}}^{c}\bar{z_{2}}^{d}|\  a+b=p,\ c+d=q\}.
\end{equation}
\end{cor}

We are now ready to compute the first and second variations of the Paneitz operator. The background space is the standard 3-sphere $S^{3}\subset C^{2}$ with the CR structure
given by the complex vector field
\begin{equation*}
Z_{\bar 1}=z_{2}\frac{\partial }{\partial\bar{z_1}}-z_{1}\frac{\partial }{\partial\bar{z_2}}.
\end{equation*}
Fix $\phi$, we use $\phi_{t}=t\phi$ to define the deformation of the CR structures along $\phi$, i.e., for each $t$, the CR structure is defined by the complex
vector field
\begin{equation}\label{va0}
Z_{\bar 1}^{t}=Z_{\bar 1}{}^{\phi_{t}}=F(Z_{\bar 1}+t\phi Z_{1}),
\end{equation}
where $F=\frac{1}{(1-t^{2}|\phi|^2)^{1/2}}$. The Kohn Laplacian for the deformed structure will be denoted by $\Box_{b}^{t}$. Then the coresponding Paneitz operator satisfies
\begin{equation}
4P_{0}^{t}=\Box_{b}^{t}\overline{\Box}_{b}^{t}-2Q^{t},
\end{equation}
where $Q^t$ is a second order differential operator defined by $Q^{t}f=2i(A^{t11}f_1)_{1}$ for each smooth function $f$, i.e.,
\begin{equation}\label{qt}
Q^t=2i\Big(A^{t11}Z_{1}^{t}Z_{1}^{t}+(Z_{1}^{t}A^{t11})Z_{1}^{t}-A^{t11}\theta_{\bar 1}^{t\bar 1}(Z_{1}^{t})Z_{1}^{t}\Big).
\end{equation}
We would like to compute the first and second variations of $4P_{0}$ and use "$\cdot$" to denote differentiation with respect to $t$.
We have the following proposition.

\begin{prop}\label{fsva}
We have
\begin{equation}\label{fva}
4\dot{P}_{0}^{t}|_{t=0}=-2D\overline{\Box}_{b}-2\Box_{b}D+4(EZ_{1}Z_{1}+E_{1}Z_{1}),
\end{equation}
and
\begin{equation}\label{sva}
\begin{split}
4\ddot{P}_{0}^{t}|_{t=0}&=16|\phi|^{2}P_{0}+2|\phi|^{2}(\Box_{b}\Box_{b}+\overline{\Box}_{b}\overline{\Box}_{b})+8D^{2}-8E\bar{\phi}\Delta_{b}+8\nabla_{b}(E\bar{\phi})\\
&+4(\Box_{b}|\phi|^{2})\Delta_{b}-8(\nabla_{b}|\phi|^{2})\Delta_{b}-4(\nabla_{b}|\phi|^{2})\overline{\Box}_{b}-4\Box_{b}(\nabla_{b}|\phi|^{2}),
\end{split}
\end{equation}
where
\begin{equation*}
\begin{split}
D&=\phi Z_{1}Z_{1}+\bar{\phi}Z_{\bar 1}Z_{\bar 1}+\phi_{1}Z_{1}+\bar{\phi}_{\bar 1}Z_{\bar 1};\\
E&=4\phi+i\phi_{0}
\end{split}
\end{equation*}
\end{prop}
\begin{proof}
The first and the second derivative with respect to $t$ are, respectively,
\begin{equation}\label{va1}
4\dot{P}_{0}^{t}=\dot{\Box}_{b}^{t}\overline{\Box}_{b}^{t}+\Box_{b}^{t}\dot{\overline{\Box}}_{b}^{t}-2\dot{Q}^{t};
\end{equation}
and

\begin{equation}\label{va2}
4\ddot{P}_{0}^{t}=\ddot{\Box}_{b}^{t}\overline{\Box}_{b}^{t}+\Box_{b}^{t}\ddot{\overline{\Box}}_{b}^{t}
+2\dot{\Box}_{b}^{t}\dot{\overline{\Box}}_{b}^{t}-2\ddot{Q}^{t}.
\end{equation}
From (\ref{coto23}) and (\ref{coto24}), we have

\begin{equation}\label{va3}
\begin{split}
\theta_{1}^{t1}(Z_{\bar 1}^t)&=\theta_{1}^{t1}(F(Z_{\bar 1}+t\phi Z_{1}))\\
&=-(F_{\bar 1}+B_{12}+t\phi F_{1}+t\phi B_{11}),
\end{split}
\end{equation}
where
\begin{equation}\label{va4}
\begin{split}
B_{11}&=F^{2}(-2F_{1}-tF\bar{\phi}_{\bar 1}-t^{2}\bar{\phi}F\phi_{1}-2t\bar{\phi}F_{\bar 1})\\
&=-t(F^{3}\bar{\phi}_{\bar 1})-t^{2}(F^{3}\bar{\phi}\phi_{1}+F^{5}(Z_{1}|\phi|^{2}))+O(t^{3}),
\end{split}
\end{equation}
and
\begin{equation}\label{va5}
\begin{split}
B_{12}&=F^{2}(2t^{2}|\phi|^{2}F_{\bar 1}+tF\phi_{1}+t^{2}F\phi\bar{\phi}_{\bar 1}+2t\phi F_{1})\\
&=t(F^{3}\phi_{1})+t^{2}(F^{3}\phi\bar{\phi}_{\bar 1})+O(t^{3}).
\end{split}
\end{equation}
Now we compute the Kohn Laplacian and its variations. We have

\begin{equation}\label{va6}
\begin{split}
\overline{\Box}_{b}^{t}&=-2Z_{\bar 1}^{t}Z_{1}^{t}+2\theta_{1}^{t1}(Z_{\bar 1}^{t})Z_{1}^{t}\\
&=-2Z_{\bar 1}^{t}Z_{1}^{t}-2(F_{\bar 1}+B_{12}+t\phi F_{1}+t\phi B_{11})Z_{1}^{t};
\end{split}
\end{equation}
thus
\begin{equation}\label{va7}
\begin{split}
\dot{\overline{\Box}}_{b}^{t}&=-2(\dot{Z}_{\bar 1}^{t}Z_{1}^{t}+Z_{\bar 1}^{t}\dot{Z}_{1}^{t})\\
&-2(\dot{F}_{\bar 1}+\dot{B}_{12}+\phi F_{1}+t\phi\dot{F}_{1}+\phi B_{11}+t\phi\dot{B}_{11})Z_{1}^{t}\\
&-2(F_{\bar 1}+B_{12}+t\phi F_{1}+t\phi B_{11})\dot{Z}_{1}^{t},
\end{split}
\end{equation}
where for any complex vector field $Z\in TS^3$, we have
\begin{equation*}
\begin{split}
ZF&=\frac{1}{2}t^{2}F^{3}(Z|\phi|^{2})\\
\dot{F}&=t|\phi|^{2}F^3\\
Z\dot{F}&=tF^{3}(Z|\phi|^{2})+\frac{3}{2}t^{3}|\phi|^{2}F^{5}(Z|\phi|^{2}),
\end{split}
\end{equation*}
and
\begin{equation*}
\dot{Z}_{\bar 1}^{t}=\dot{F}(Z_{\bar 1}+t\phi Z_{1})+F\phi Z_{1}.
\end{equation*}
Next, we would like to expand $-2Q^t$ with respect to $t$ at $t=0$. Denote $4\phi+i\phi_{0}$ as $E$. From (\ref{coto23}), we see that
\begin{equation*}
A^{t1}{}_{\bar 1}=F^{2}(4it\phi-t\phi_{0})=itF^{2}E,
\end{equation*}
hence, from (\ref{qt}),
\begin{equation}\label{va10}
\begin{split}
-2Q^{t}&=-4i\Big(A^{t11}Z_{1}^{t}Z_{1}^{t}+(Z_{1}^{t}A^{t11})Z_{1}^{t}-A^{t11}\theta_{\bar 1}^{t\bar 1}(Z_{1}^{t})Z_{1}^{t}\Big)\\
&=-4i(A^{t11}Z_{1}^{t}Z_{1}^{t}+(Z_{1}^{t}A^{t11})Z_{1}^{t})-4iA^{t11}(F_{1}+\overline{B_{12}}+t\bar{\phi} F_{\bar 1}+t\bar{\phi}\overline{B_{11}})Z_{1}^{t}\\
&=4tF^{2}EZ_{1}^{t}Z_{1}^{t}+(Z_{1}^{t}(4tF^{2}E))Z_{1}^{t}+4tF^{2}E(F_{1}+\overline{B_{12}}+t\bar{\phi} F_{\bar 1}+t\bar{\phi}\overline{B_{11}})Z_{1}^{t},
\end{split}
\end{equation}
where

\begin{equation}\label{va11}
\begin{split}
4tF^{2}EZ_{1}^{t}Z_{1}^{t}&=4tF^{2}EF(Z_{1}+t\bar{\phi}Z_{\bar 1})\big(F(Z_{1}+t\bar{\phi}Z_{\bar 1})\big)\\
&=4F^{4}E(tZ_{1}Z_{1}-t^{2}\bar{\phi}\Delta_{b}+t^{2}\bar{\phi}_{1}Z_{\bar 1})+O(t^3);\\
(Z_{1}^{t}(4tF^{2}E))Z_{1}^{t}&=4tF^{2}\big((Z_{1}+t\bar{\phi}Z_{\bar 1})(F^{2}E)\big)(Z_{1}+t\bar{\phi}Z_{\bar 1})\\
&=4t(F^{4}E_{1})Z_{1}+4t^{2}F^{4}(\bar{\phi}E_{\bar 1}Z_{1}+\bar{\phi}E_{1}Z_{\bar 1})+O(t^3);\\
\end{split}
\end{equation}
and
\begin{equation}\label{va12}
\begin{split}
&F_{1}+\overline{B_{12}}+t\bar{\phi} F_{\bar 1}+t\bar{\phi}\overline{B_{11}}\\
=&\frac{1}{2}t^{2}F^{3}(Z_{1}|\phi|^{2})+tF^{3}\bar{\phi}_{\bar 1}+t^{2}(F^{3}\bar{\phi}\phi_{1})+O(t^3)-t^{2}F^{3}\phi_{1}\bar{\phi}\\
=&\frac{1}{2}t^{2}F^{3}(Z_{1}|\phi|^{2})+tF^{3}\bar{\phi}_{\bar 1}+O(t^3).
\end{split}
\end{equation}
Substituting (\ref{va11}) and (\ref{va12}) into (\ref{va10}), we get

\begin{equation}\label{va13}
\begin{split}
-2Q^{t}&=4tF^{4}(EZ_{1}Z_{1}+E_{1}Z_{1})+4t^{2}F^{4}E(-\bar{\phi}\Delta_{b}+\bar{\phi}_{1}Z_{\bar 1})\\
&+4t^{2}F^{4}(\bar{\phi}E_{\bar 1}Z_{1}+\bar{\phi}E_{1}Z_{\bar 1})+4t^{2}F^{6}E\bar{\phi}_{\bar 1}Z_{1}+O(t^3).
\end{split}
\end{equation}
Therefore, from (\ref{va1}) together with (\ref{va4}), (\ref{va5}), (\ref{va6}), (\ref{va7}) and (\ref{va13}), we get
\begin{equation*}
\begin{split}
\overline{\Box}_{b}^{t}|_{t=0}&=\overline{\Box}_{b};\\
\dot{\overline{\Box}}_{b}^{t}|_{t=0}&=-2(\phi Z_{1}Z_{1}+\bar{\phi}Z_{\bar 1}Z_{\bar 1}+\bar{\phi}_{\bar 1}Z_{\bar 1}+\phi_{1}Z_{1});\\
-2\dot{Q}^{t}|_{t=0}&=4(EZ_{1}Z_{1}+E_{1}Z_{1}),
\end{split}
\end{equation*}
hence the first variation of $4P_{0}$:

\begin{equation}
4\dot{P}_{0}^{t}|_{t=0}=-2D\overline{\Box}_{b}-2\Box_{b}D+4(EZ_{1}Z_{1}+E_{1}Z_{1}),
\end{equation}
where
\begin{equation*}
D=\phi Z_{1}Z_{1}+\bar{\phi}Z_{\bar 1}Z_{\bar 1}+\phi_{1}Z_{1}+\bar{\phi}_{\bar 1}Z_{\bar 1}.
\end{equation*}
Finally, for the second variation of Paneitz operator, we also need to compute the second variation of the Kohn Laplacian.
From (\ref{va7}), taking the derivative with respect to t, we get

\begin{equation}\label{va14}
\begin{split}
\ddot{\overline{\Box}}_{b}^{t}&=-2(\ddot{Z}_{\bar 1}^{t}Z_{1}^{t}+Z_{\bar 1}^{t}\ddot{Z}_{1}^{t}+2\dot{Z}_{\bar 1}^{t}\dot{Z}_{1}^{t})\\
&-2(\ddot{F}_{\bar 1}+\ddot{B}_{12}+2\phi\dot{F}_{1}+t\phi\ddot{F}_{1}+2\phi\dot{B}_{11}+t\phi\ddot{B}_{11})Z_{1}^{t}\\
&-4(\dot{F}_{\bar 1}+\dot{B}_{12}+\phi F_{1}+t\phi\dot{F}_{1}+\phi B_{11}+t\phi\dot{B}_{11})\dot{Z}_{1}^{t}\\
&-2(F_{\bar 1}+B_{12}+t\phi F_{1}+t\phi B_{11})\ddot{Z}_{1}^{t},
\end{split}
\end{equation}
where, for all $Z\in TS^3$,
\begin{equation*}
\begin{split}
\ddot{Z}_{\bar 1}^{t}&=\ddot{F}(Z_{\bar 1}+t\phi Z_{1})+2\dot{F}\phi Z_{1};\\
\ddot{F}&=|\phi|^{2}F^{3}+3t^{2}|\phi|^{4}F^{5};\\
Z\ddot{F}&=F^{3}(Z|\phi|^{2})+\frac{15}{2}t^{2}|\phi|^{2}F^{5}(Z|\phi|^{2})+\frac{15}{2}t^{4}|\phi|^{4}F^{7}(Z|\phi|^{2}),
\end{split}
\end{equation*}
which implies that

\begin{equation}\label{va15}
\ddot{\overline{\Box}}_{b}^{t}|_{t=0}=2|\phi|^{2}(\overline{\Box}_{b}+\Box_{b})-4(Z_{1}|\phi|^{2})Z_{\bar 1}-4(Z_{\bar 1}|\phi|^{2})Z_{1}.
\end{equation}
Therefore, from (\ref{va2}) together with (\ref{va4}), (\ref{va5}), (\ref{va6}), (\ref{va7}), (\ref{va13}) and (\ref{va15}), we get
\begin{equation*}
\begin{split}
4\ddot{P}_{0}^{t}|_{t=0}&=2|\phi|^{2}(\overline{\Box}_{b}+\Box_{b})\overline{\Box}_{b}+\Box_{b}(2|\phi|^{2}(\overline{\Box}_{b}+\Box_{b}))\\
&-4((Z_{\bar 1}|\phi|^{2})Z_{1}+(Z_{1}|\phi|^{2})Z_{\bar 1})\overline{\Box}_{b}-4\Box_{b}((Z_{\bar 1}|\phi|^{2})Z_{1}+(Z_{1}|\phi|^{2})Z_{\bar 1})\\
&+8D^{2}+8E(-\bar{\phi}\Delta_{b})+8\bar{\phi}(E_{\bar 1}Z_{1}+E_{1}Z_{\bar 1})+8E(\bar{\phi}_{\bar 1}Z_{1}+\bar{\phi}_{1}Z_{\bar 1}),
\end{split}
\end{equation*}
where
\begin{equation*}
\begin{split}
&\Box_{b}(2|\phi|^{2}(\overline{\Box}_{b}+\Box_{b}))\\
=&4\Box_{b}(|\phi|^{2}\Delta_{b})\\
=&4[(\Box_{b}|\phi|^{2})\Delta_{b}+|\phi|^{2}\Box_{b}\Delta_{b}-2(|\phi|^{2}_{\bar 1}Z_{1}\Delta_{b}+|\phi|^{2}_{1}Z_{\bar 1}\Delta_{b})]\\
=&2|\phi|^{2}\Box_{b}(\overline{\Box}_{b}+\Box_{b})+4(\Box_{b}|\phi|^{2})\Delta_{b}-8(\nabla_{b}|\phi|^{2})\Delta_{b},
\end{split}
\end{equation*}
hence
\begin{equation}\label{secvar}
\begin{split}
4\ddot{P}_{0}^{t}|_{t=0}&=16|\phi|^{2}P_{0}+2|\phi|^{2}(\Box_{b}\Box_{b}+\overline{\Box}_{b}\overline{\Box}_{b})+8D^{2}-8E\bar{\phi}\Delta_{b}+8\nabla_{b}(E\bar{\phi})\\
&+4(\Box_{b}|\phi|^{2})\Delta_{b}-8(\nabla_{b}|\phi|^{2})\Delta_{b}-4(\nabla_{b}|\phi|^{2})\overline{\Box}_{b}-4\Box_{b}(\nabla_{b}|\phi|^{2}).
\end{split}
\end{equation}
This completes the proposition.
\end{proof}

For the reader's convenience we list the following useful facts \cite{C} that are necessary for the subsequent computations. We recall $H_{p,q}$
denotes the space of bi-graded spherical harmonics of type $(p,q)$ on $S^3$. Then for $f\in H_{p,q}$, and the operators associated to the standard
structure on $S^3$,
\begin{equation}\label{eigeq}
\begin{split}
\Box_{b}f=2(p+1)qf,\ \ \ &\overline{\Box}_{b}f=2(q+1)pf,\\
P_{0}f=pq(p+1)(q+1)f,\ \ \ &\Delta_{b}f=-(f_{1\bar 1}+f_{\bar{1}1})=(2pq+p+q)f.
\end{split}
\end{equation}

\begin{prop}\label{apfvf}
Let $\phi\in C^{\infty}(S^3)$. Let $g_{\bar 1}=0$ and $f\in H_{p,0}$ or $H_{0,p}$. Then for any $p\geq 0$
\begin{equation}
<\dot{P}_{0}^{t}|_{t=0}f, g>\equiv 0.
\end{equation}
\end{prop}
\begin{proof}
We only display the proof for $f\in H_{p,0}$. The proof for $f\in H_{0,p}$ is similar. We repeatedly integrate by parts and use $g_{\bar 1}=0$ to finish the proof. From the first variation (\ref{fva}), we have
\begin{equation}
\begin{split}
4<\dot{P}_{0}^{t}|_{t=0}f, g>&=-2<D\overline{\Box}_{b}f, g>-2<Df, \Box_{b}g>\\
&+4<Ef_{11},g>+4<E_{1}f_{1}, g>.
\end{split}
\end{equation}
Integration by parts in the fourth term, and using $g_{\bar 1}=0$ in the integration by parts and in the second term yields
\begin{equation}
\begin{split}
4<\dot{P}_{0}^{t}|_{t=0}f, f>&=-2<D\overline{\Box}_{b}f, g>,\\
&=-4p(q+1)<Df, g>=-4p(q+1)<\phi f_{11}+\phi_{1}f_{1},g>\\
&=0,
\end{split}
\end{equation}
where the last equality is due to the integration by parts and using $g_{\bar 1}=0$.
\end{proof}

\begin{prop}\label{apfvf1}
Let $\phi\in C^{\infty}(S^3)$. Let $f\in {\bf H}=C^{\infty}(S^3)\cap
\oplus_{p\geq 1} H_{p,0}\oplus H_{0,p}$. Then
$$<\dot{P}_{0}^{t}|_{t=0}f,f>\equiv 0.$$
\end{prop}
\begin{proof}
The Proposition follows from Proposition \ref{apfvf} and the fact that the operator $\dot{P}_{0}^{t}|_{t=0}$ is real.
\end{proof}

\begin{prop}\label{dsp}
Let $\phi\in C^{\infty}(S^3)$. Then
\begin{equation}
<D^{2}f,f>\ \geq 0,\ \ \textrm{for all}\ f\in KerP_{0}
\end{equation}
\end{prop}

\begin{proof}
Since $f\in KerP_{0}$, its Fourier representation has the form
\[f=\sum_{p,q=0}^{\infty}f_{pq},\ \ \textrm{with}\ p=0\ \textrm{or}\ q=0.\]
Thus we divide it into a CR holomorphic part and a anti-CR holomorphic part, that is, $f$ has an expression $f=u+v$, where $u$ and $v$ is a CR function and
anti-CR function, respectively. This means $u_{\bar 1}=0$ and $v_{1}$=0.
Now we compute
\begin{equation}\label{d2ff}
<D^{2}f,f>=<D^{2}u,u>+<D^{2}v,v>+<D^{2}u,v>+<D^{2}v,u>.
\end{equation}

Since $u_{\bar 1}=0$, we get $Du=\phi u_{11}+\phi_{1}u_{1}$. Thus,
\begin{equation}\label{1}
\begin{split}
<D^{2}u, u>&=<\phi(\phi u_{11}+\phi_{1}u_{1})_{11},u>+<\bar{\phi}(\phi u_{11}+\phi_{1}u_{1})_{\bar{1}\bar 1},u>\\
&+<\phi_{1}(\phi u_{11}+\phi_{1}u_{1})_{1},u>+<\bar{\phi}_{\bar 1}(\phi u_{11}+\phi_{1}u_{1})_{\bar 1},u>.
\end{split}
\end{equation}
Integrate by parts the last two terms to get
\begin{equation}\label{2}
\begin{split}
&<\phi_{1}(\phi u_{11}+\phi_{1}u_{1})_{1}, u>+<\bar{\phi}_{\bar 1}(\phi u_{11}+\phi_{1}u_{1})_{\bar 1}, u>\\
=-&<\phi(\phi u_{11}+\phi_{1}u_{1})_{11},u>-<\phi(\phi u_{11}+\phi_{1}u_{1})_{1}, u_{\bar 1}>\\
-&<\bar{\phi}(\phi u_{11}+\phi_{1}u_{1})_{\bar{1}\bar 1}, u>-<\bar{\phi}(\phi u_{11}+\phi_{1}u_{1})_{\bar 1}, u_{1}>.
\end{split}
\end{equation}
Taking together (\ref{1}) and (\ref{2}), and using $u_{\bar 1}=0$ and integrating by parts again, we see
\begin{equation}\label{d2uu}
\begin{split}
<D^{2}u, u>&=-\int_{S^3}\bar{\phi}(\phi u_{11}+\phi_{1}u_{1})_{\bar 1}\bar{u}_{\bar 1}\\
&=\int_{S^3}|\phi|^{2}|u_{11}|^{2}+|\phi_{1}|^{2}|u_{1}|^{2}+\int_{S^3}(\bar{\phi}\phi_{1}u_{1}\bar{u}_{\bar{1}\bar 1}+\bar{\phi}_{\bar 1}\phi u_{11}\bar{u}_{\bar 1})\\
&=\int_{S^3}|\phi u_{11}+\phi_{1}u_{1}|^{2}\ \geq 0.
\end{split}
\end{equation}
Similarly, using $v_{1}=0$, we get
\begin{equation}\label{d2uv}
<D^{2}u, v>=\int_{S^3}(\phi u_{11}+\phi_{1}u_{1})(\phi \bar{v}_{11}+\phi_{1}\bar{v}_{1}).
\end{equation}
Using the conjugate and the fact $D^{2}$ is real, we see
\begin{equation}\label{d2vv}
<D^{2}v,v>=\overline{<D^{2}\bar{v},\bar{v}>}=\int_{S^3}|\phi \bar{v}_{11}+\phi_{1}\bar{v}_{1}|^{2},
\end{equation}
and
\begin{equation}\label{d2vu}
<D^{2}v, u>=\int_{S^3}\overline{(\phi u_{11}+\phi_{1}u_{1})}\overline{(\phi \bar{v}_{11}+\phi_{1}\bar{v}_{1})}.
\end{equation}
Substituting (\ref{d2uu}), (\ref{d2uv}), (\ref{d2vv}) and (\ref{d2vu}) into (\ref{d2ff}), we get
\[<D^{2}f,f>=\int_{S^3}|(\phi u_{11}+\phi_{1}u_{1})+(\bar{\phi}v_{\bar{1}\bar1}+\bar{\phi}_{\bar1}{v}_{\bar 1})|^{2}\geq 0.\]
This finishes the proof.
\end{proof}

Using (\ref{secvar}) write

\begin{equation}\label{deofse}
4<\ddot{P}_{0}^{t}|_{t=0}f,f>=8<D^{2}f,f>+<Rf,f>.
\end{equation}

\begin{prop}\label{apsvf}
Let $\phi\in C^{\infty}(S^3)$. Let $f\in H_{p,0}$ or $f\in H_{0,p},$ and $g_{\bar 1}=0$. Then for all $p\geq 0,$ we have
\begin{equation*}
(a) <Rf,g>=0,\ \ \textrm{if}\ f\in H_{0,p},
\end{equation*}
and
\begin{equation*}
(b) <Rf,g>=8\int_{S^3}(p|\phi|^{2}-E\bar{\phi})f_{1}\bar{g}_{\bar 1}\ \ \textrm{if} \ f\in H_{p,0}.
\end{equation*}
\end{prop}

\begin{proof}

We first compute each term in the formula of the second variation of the Paneitz operator (see (\ref{sva})). For all $f\in H_{p,0}$,
and $g_{\bar 1}=0$.

By (\ref{eigeq}),

\begin{equation}\label{3}
<16|\phi|^{2}P_{0}f,g>=0.
\end{equation}

\begin{equation}\label{4}
<2|\phi|^{2}(\Box_{b}\Box_{b}+\overline{\Box}_{b}\overline{\Box}_{b})f,g>=8p^{2}\int_{S^3}|\phi|^{2}f\bar{g}.
\end{equation}

\begin{equation}\label{5}
<-8E\bar{\phi}\Delta_{b}f,g>=-8p\int_{S^3}E\bar{\phi}f\bar{g}.
\end{equation}
Integrating by parts gives
\begin{equation}\label{6}
\begin{split}
<8\nabla_{b}(E\bar{\phi})f,g>&=8<(E\bar{\phi})_{\bar 1}f_{1}+(E\bar{\phi})_{1}f_{\bar 1},g>=8<(E\bar{\phi})_{\bar 1}f_{1},g>\\
&=-8<E\bar{\phi}f_{1\bar 1},g>-8<E\bar{\phi}f_{1},g_{1}>\\
&=8p\int_{S^3}E\bar{\phi}f\bar{g}-8\int_{S^3}E\bar{\phi}f_{1}\bar{g}_{\bar 1}.
\end{split}
\end{equation}

\begin{equation}\label{7}
\begin{split}
<4(\Box_{b}|\phi|^{2})\Delta_{b}f,g>&=4p<(\Box_{b}|\phi|^{2})f,g>\\
&=4p<\Box_{b}(|\phi|^{2}f)-|\phi|^{2}\Box_{b}f+2(|\phi|^{2}_{\bar 1}f_{1}+|\phi|^{2}_{1}f_{\bar 1}),g>\\
&=8p<(\nabla_{b}|\phi|^{2})f,g>\\
&=<8(\nabla_{b}|\phi|^{2})\Delta_{b}f,g>.
\end{split}
\end{equation}

\begin{equation}\label{8}
\begin{split}
<-4(\nabla_{b}|\phi|^{2})\overline{\Box}_{b}f,g>&=-8p<(\nabla_{b}|\phi|^{2})f,g>=-8p<|\phi|^{2}_{\bar 1}f_{1},g>\\
&=8p<|\phi|^{2}f_{1\bar 1},g>+8p<|\phi|^{2}f_{1},g_{1}>\\
&=8p\int_{S^3}|\phi|^{2}f_{1}\bar{g}_{\bar 1}-8p^{2}\int_{S^3}|\phi|^{2}f\bar{g}.
\end{split}
\end{equation}

\begin{equation}\label{9}
<-4\Box_{b}(\nabla_{b}|\phi|^{2})f,g>=<-4(\nabla_{b}|\phi|^{2})f,\Box_{b}g>=0\\
\end{equation}
We collect similar terms from (\ref{3}) to (\ref{9}) and observe that both the coefficients of terms
\[\int_{S^3}|\phi|^{2}f\bar{g}\ \ \textrm{and}\ \ \int_{S^3}E\bar{\phi}f\bar{g}\]
are zero. The coefficient of the term $\int_{S^3}|\phi|^{2}f_{1}\bar{g}_{\bar 1}$ is $8p$
and The coefficient of the term $\int_{S^3}E\bar{\phi}f_{1}\bar{g}$ is $-8$. This completes the proof of $(b)$.

Similarly, for $f\in H_{0,p}$, integrating by parts and using $f_{1}=0$ and $g_{\bar 1}=0$, we get
\begin{equation}
\begin{split}
<16|\phi|^{2}P_{0}f,g>&=0;\\
<2|\phi|^{2}(\Box_{b}\Box_{b}+\overline{\Box}_{b}\overline{\Box}_{b})f,g>&=8p^{2}\int_{S^3}|\phi|^{2}f\bar{g};\\
<-8E\bar{\phi}\Delta_{b}f,g>&=-8p\int_{S^3}E\bar{\phi}f\bar{g};\\
<8\nabla_{b}(E\bar{\phi})f,g>&=8p\int_{S^3}E\bar{\phi}f\bar{g};\\
<4(\Box_{b}|\phi|^{2})\Delta_{b}f,g>&=<8(\nabla_{b}|\phi|^{2})\Delta_{b}f,g>-8p^{2}\int_{S^3}|\phi|^{2}f\bar{g};\\
<-4(\nabla_{b}|\phi|^{2})\overline{\Box}_{b}f,g>&=0;\\
<-4\Box_{b}(\nabla_{b}|\phi|^{2})f,g>&=<-4(\nabla_{b}|\phi|^{2})f,\Box_{b}g>=0,
\end{split}
\end{equation}
Taking all together these terms, we get $<Rf,g>=0$. This completes the proof of $(a)$.
\end{proof}

\begin{proof}[{\bf The proof of Proposition \ref{main5}:}]
From (\ref{deofse}),
\[<\ddot{P}_{0}^{t}f,f>=2<D^{2}f,f>+\frac{1}{4}<Rf,f>.\]
Using Proposition \ref{dsp},
\[<\ddot{P}_{0}^{t}f,f>\geq \frac{1}{4}<Rf,f>.\]
Now,
\[f=\sum_{k\geq 1}f^{k}+\sum_{k\geq 1}g^{k},\ \ \ f^{k}\in H_{k,0},\ \ g^{k}\in H_{0,k}.\]
Using the relations in Proposition \ref{apsvf}, one computes
\[\frac{1}{4}<Rf,f>=2\sum_{k,l}\int_{S^3}(k|\phi|^{2}-E\bar{\phi})f^{k}_{1}\overline{f^{l}_{1}}+2\sum_{k,l}\int_{S^3}(k|\phi|^{2}-E\bar{\phi})g^{k}_{\bar 1}\overline{g^{l}_{\bar 1}}.\]
This ends the proof.
\end{proof}

 We are going to recall the Hopf fibration of $S^3\subset C^{2}$.
We consider the space $BC^{2}$ by blowing up the orign from $C^{2}$. Let $(z_{1},z_{2})$ be the linear coordinates on $C^2$ and $(\zeta, w)$ denote
blow up coordinates on $BC^2$. These coordinates are related by
\[\zeta=z_{1},\ \ \ w=\frac{z_{2}}{z_{1}}.\]
 $BC^2$ is the tautological line bundle over $CP^1$. We see that $w$ is an affine coordinate on $CP^1=S^2$, which is the blow up of the origin, and
$\zeta$ is the fiber coordinate. Now $S^3$ is defined by $r(z_{1},z_{2})=|z_{1}|^{2}+|z_{2}|^{2}=1$, so we have
\begin{equation}
\begin{split}
1=&r=|z_{1}|^{2}+|z_{2}|^{2}\\
&=|\zeta|^{2}(1+|w|^{2})\\
&=|\zeta|^{2}e^{H(w)},
\end{split}
\end{equation}
where $H(w)=\ln{(1+|w|^{2})}$, defines a circle bundle over $CP^{1}=S^2$. This is a fibration of $S^3$. Let $\theta$ be the standard contact form on $S^3$.
Then we have

\begin{equation}\label{jac}
\theta\wedge d\theta=i\frac{\partial^{2}H(w)}{\partial w\partial \bar{w} }d\psi\wedge dw\wedge d\bar{w},
\end{equation}
where $\zeta=|\zeta|e^{i\psi}$ and $\psi$ is the fiber coordinate. Also, we see that $T$ is the generator of the circular action
 $(z_{1},z_{2})\rightarrow (e^{i\psi}z_{1}, e^{i\psi}z_{2})$ with period $2\pi$, that is,
\[T=\frac{\partial }{\partial \psi}.\]

\begin{prop}\label{exse}
For $f\in H_{p,0},\ f_{1}=e^{i(p-2)\psi}H(w,\bar{w})$ and $g\in H_{0,p}\ g_{\bar 1}=e^{-i(p-2)\psi}G(w,\bar{w})$, so in particular
both $|f_{1}|$ and $|g_{\bar 1}|$ does not depend on the fiber coordinate $\psi$.\\
 For $\phi\in C^{\infty}(S^3)$, we have the following Fourier expansion
of $\phi$ with respect to $\psi$
\[\phi=\sum_{p,q=0}^{\infty}\phi_{pq}=\sum_{m\in Z}e^{im\psi}\phi_{m}(w,\bar{w}),\]
where $m=p-q$ and $\phi_{m}$ is a function only defined on $S^2$.
\end{prop}
\begin{proof}
If $f\in H_{p,0}$, say, $f=\sum_{a+b=p}c_{ab}z_{1}^{a}z_{2}^{b}$ then
\begin{equation}
\begin{split}
f_{1}&=Z_{1}f=(\bar{z}_{2}\frac{\partial }{\partial z_{1}}-\bar{z}_{1}\frac{\partial }{\partial z_{2}})f\\
&=\sum_{a+b=p}c_{ab}(az_{1}^{a-1}z_{2}^{b}\bar{z}_{2}-bz_{1}^{a}z_{2}^{b-1}\bar{z}_{1})\\
&=\sum_{a+b=p}c_{ab}(a\zeta^{a+b-1}\bar{\zeta}w^{b}\bar{w}-b\zeta^{a+b-1}\bar{\zeta}w^{b-1})\\
&=\zeta^{a+b-1}\bar{\zeta}\left(\sum_{a+b=p}c_{ab}(aw^{b}\bar{w}-bw^{b-1})\right).
\end{split}
\end{equation}
Since on $S^3,\ 1=|\zeta|^{2}e^{H(w)}$, we have
\begin{equation}
\begin{split}
|f_{1}|=&|\zeta|^{a+b}\left|\sum_{a+b=p}c_{ab}(aw^{b}\bar{w}-bw^{b-1})\right|\\
&=e^{-\frac{a+b}{2}H(w)}\left|\sum_{a+b=p}c_{ab}(aw^{b}\bar{w}-bw^{b-1})\right|,
\end{split}
\end{equation}
which does not depend on $\psi$.
On the other hand, for $\phi\in C^{\infty}(S^3)$, it has the Fourier representation $\phi=\sum_{p,q=0}^{\infty}\phi_{pq}$. We would like to express it by the coordinates
$(\zeta, w)$. We denote $\phi_{pq}=\sum_{a+b=p,c+d=q}c_{abcd}z_{1}^{a}z_{2}^{b}\bar{z}_{1}^{c}\bar{z}_{2}^{d}$. Then
\begin{equation}
\begin{split}
\phi_{pq}&=\sum_{a+b=p,c+d=q}c_{abcd}\zeta^{a+b}w^{b}\bar{\zeta}^{c+d}\bar{w}^{d}\\
&=\sum_{a+b=p,c+d=q}c_{abcd}e^{i(p-q)\psi}|\zeta|^{p+q}w^{b}\bar{w}^{d}\\
&=\sum_{a+b=p,c+d=q}c_{abcd}e^{i(p-q)\psi}e^{\frac{-(p+q)H(w)}{2}}w^{b}\bar{w}^{d}\\
&=e^{im\psi}\phi_{m}(w,\bar{w}),
\end{split}
\end{equation}
where $m=p-q$ and
$\phi_{m}=\left(\sum_{a+b=p,c+d=q}c_{abcd}e^{\frac{-(p+q)H(w)}{2}}w^{b}\bar{w}^{d}\right)$.
\end{proof}

\begin{prop}\label{keyf}
(a) Let $\phi\in C^{\infty}(S^3)$, then for $\phi$ satisfying (BE)
\[\int_{S^3}(k|\phi|^{2}-E\bar{\phi})|f^{k}_{1}|^{2}\geq \int_{S^3}|\phi|^{2}|f^{k}_{1}|^{2},\ \ k\geq 1;\]
\[\int_{S^3}(k|\phi|^{2}-E\bar{\phi})|g^{k}_{\bar 1}|^{2}\geq \int_{S^3}|\phi|^{2}|g^{k}_{\bar 1}|^{2},\ \ k\geq 1.\]
(b) For any $\phi\in P_{p_{1},q_{1}}$,
\begin{equation*}
\int_{S^3}(k|\phi|^{2}-E\bar{\phi})f^{k}_{1}\overline{f^{l}_{1}}=\left\{ \begin{array}{ll}
0&,k\neq l\\
\int_{S^3}(k+p_{1}-q_{1}-4)|\phi|^{2}|f^{k}_{1}|^{2}&,k=l.
\end{array}
\right.\end{equation*}
\begin{equation*}
\int_{S^3}(k|\phi|^{2}-E\bar{\phi})g^{k}_{\bar 1}\overline{g^{l}_{\bar 1}}=\left\{ \begin{array}{ll}
0&,k\neq l\\
\int_{S^3}(k+p_{1}-q_{1}-4)|\phi|^{2}|g^{k}_{\bar 1}|^{2}&,k=l.
\end{array}
\right.\end{equation*}
\end{prop}
\begin{proof}
We only display the proof for the first statement of part (a). From
Proposition \ref{exse}, $|f^{k}_{1}|$ is independent of the fiber
variable $\psi$. Changing variables using (\ref{jac}), we have using
the Fourier expansion of $\phi$ in Proposition \ref{exse} that
\[\phi=\sum\phi_{pq}=\sum_{m\in Z}e^{im\psi}\phi_{m}(w,\bar{w}),\ \ m=p-q.\]
Now
\[E=4\phi+i\phi_{0}.\]
Thus
\[E\bar{\phi}=(\sum_{m\in Z}e^{im\psi}(4-m)\phi_{m}(w,\bar{w}))\bar{\phi}.\]
So by Plancherel's theorem,
\begin{equation}\label{planch}
\begin{split}
\int_{S^3}(k|\phi|^{2}-E\bar{\phi})|f^{k}_{1}|^{2}&=\int_{S^2}\left(\int_{S^1}(k|\phi|^{2}-E\bar{\phi})\right)|f^{k}_{1}|^{2}\\
&=\int_{S^2}\left(\sum_{m}(k+m-4)|\phi_{m}|^{2}\right)|f^{k}_{1}|^{2}.
\end{split}
\end{equation}
The (BE) condition implies $m-4=p-q-4\geq 0,$ and since $k\geq 1$, the term above is bounded below by
\[\int_{S^2}\left(\sum_{m}|\phi_{m}|^{2}\right)|f^{k}_{1}|^{2}.\]
Using Plancherel's theorem again we obtain our result.

We only consider the case for $f^{k}_{1}$ in part (b), the case for $g^{k}_{\bar 1}$ is similar. Observe that if $\phi\in P_{p_{1},q_{1}}$, then
$|\phi|$ is independent of the fiber variable $\psi$, and since,
\begin{equation*}
\begin{split}
E\bar{\phi}&=(4\phi+i\phi_{0})=4|\phi|^{2}+i\phi_{0}\bar{\phi}\\
&=4|\phi|^{2}+(q_{1}-p_{1})|\phi|^{2}=(4+q_{1}-p_{1})|\phi|^{2}.
\end{split}
\end{equation*}
$E\bar{\phi}$ is also independent of the fiber variable $\psi$, and only depends on $w,\bar{w}$.Thus if $k\neq l$, the integrand may be written as
\[e^{\pm i(k-l)\psi}G_{\pm }(w,\bar{w}),\]
from which it immediately follows that if  $k\neq l$, the integral vanishes. When $k=l$, from the computation (\ref{planch}) in part (a), ,
\begin{equation*}
\begin{split}
\int_{S^3}(k|\phi|^{2}-E\bar{\phi})|f^{k}_{1}|^{2}&=(k+p_{1}-q_{1}-4)\sum_{m=p_{1}-q_{1}}\int_{S^2}|\phi_{m}|^{2}|f^{k}_{1}|^{2}\\
&=(k+p_{1}-q_{1}-4)\int_{S^3}|\phi|^{2}|f^{k}_{1}|^{2}.
\end{split}
\end{equation*}
We have our conclusion.
\end{proof}

\begin{proof}[{\bf The proof of Proposition \ref{main6}:}]
We put together Proposition \ref{main5} and the computation of the integrals in the right side of Proposition \ref{main5} which are done in Proposition \ref{keyf}.
The Proposition follows.
\end{proof}

We emphasize that in Proposition \ref{main6} we do not hypothesize that
$\phi$ satisfies (BE). The following corollaries are therefore
immediate consequences of Proposition \ref{main6}.

\begin{cor}
Let $\phi\in P_{p_{1},q_{1}}$. Then $<\ddot{P}_{0}^{t}|_{t=0}f,f>$ is positive for all $f\in H_{p,0}$ or $f\in H_{0,p},\ p\geq 1$ if $\phi$ satisfies condition (BE),i.e.,
$p_{1}\geq 4+q_{1}$.
\end{cor}

\begin{cor}
Let $\phi\in P_{p_{1},q_{1}}$ and $W$ denote the subspace of $\oplus H_{p,0}\oplus H_{0,p}$ on which $\ddot{P}_{0}^{t}|_{t=0}<0$. Then $W\subset\oplus H_{p,0}\oplus H_{0,p},$
for $p<q_{1}+4-p_{1}$.
\end{cor}

\begin{rem}
The reader may verify by making an explicit calculation for $\phi_{t}=t$, using the above formula, that is in Rossi's example, we have
\[<\ddot{P}_{0}^{t}|_{t=0}f,f>\ <0,\ \ \textrm{for}\ \ f=z_{1,}\ z_{2},\ \bar{z}_{1},\ \bar{z}_{2}.\]
\end{rem}

\bibliographystyle{plain}

\end{document}